\documentclass[11pt]{amsart}

\usepackage{euscript}
\usepackage{amsmath}
\usepackage{tikz}
\usetikzlibrary{arrows,matrix,positioning}
\usepackage{amsthm}
\usepackage{epsfig}
\usepackage{amssymb}\usepackage{amscd}
\usepackage{epic}
\usepackage{eufrak}

\numberwithin{equation}{section}
\numberwithin{equation}{subsection}

\theoremstyle{plain}
\newtheorem{theorem}[equation]{Theorem}
\newtheorem{lemma}[equation]{Lemma}
\newtheorem{proposition}[equation]{Proposition}

\newtheorem{corollary}[equation]{Corollary}

\theoremstyle{definition}
\newtheorem{example}[equation]{Example}
\newtheorem{remark}[equation]{Remark}

\newtheorem{definition}[equation]{Definition}

\newtheorem{bekezdes}[equation]{}

\numberwithin{equation}{section}
\numberwithin{equation}{subsection}

\DeclareMathOperator{\Hom}{{\rm Hom}}

\DeclareMathOperator{\rank}{{\rm rank}}

\newcommand{\et}{{\mathcal T}}
\newcommand{\bH}{{\mathbb H}}

\def\C{\mathbb C}
\def\Q{\mathbb Q}
\def\R{\mathbb R}
\def\Z{\mathbb Z}

\def\bS{\mathbb S}

\newcommand{\cale}{{\mathcal E}}

\newcommand{\calt}{{\mathcal T}}

\newcommand{\calj}{{\mathcal J}}
\newcommand{\ocalj}{\overline{{\mathcal J}}}\newcommand{\scalj}{{\mathcal J}^*}
\newcommand{\calQ}{{\mathcal Q}}
\newcommand{\calF}{{\mathcal F}}
\newcommand{\calC}{{\mathcal C}}
\newcommand{\calS}{{\mathcal S}}\newcommand{\cals}{{\bf s}}\newcommand{\calst}{{\bf st}}
\newcommand{\frs}{{\frak s}}\newcommand{\frI}{{\frak I}}
\newcommand{\ocalS}{\overline{{\mathcal S}}}
\newcommand{\calH}{\mathcal{H}}
\newcommand{\calP}{\mathcal{P}}

\newcommand{\K}{k_{can}}

\newcommand{\labelpar}{\label}
\newcommand{\m}{\mathfrak{m}}

\newcommand{\vasi}{\mathbf{i}}
\newcommand{\vast}{\mathbf{t}}

\newcommand{\lk}{l'_{[k]}}

\newcommand{\ii}{({\bf i},\overline{I})}

\newcommand{\iij}{({\bf i},\overline{J})}
\newcommand{\bt}{{\bf t}}

\newcommand{\frsw}{\mathfrak{sw}}

\title{Reduction theorem for lattice cohomology}

\author{Tam\'as L\'aszl\'o}
\address{Central European University, Budapest and A. R\'enyi Institute of Mathematics, 1053 Budapest,
Re\'altanoda u. 13-15,  Hungary.}
\email{laszlo.tamas@renyi.mta.hu}
\thanks{The first author is supported by the PhD program of the CEU, Budapest and by the `Lend\"ulet' and ERC program `LTDBud' at R\'enyi
Institute. The second author is partially supported by OTKA Grant 100796.}

\author{Andr\'as N\'emethi}
\address{A. R\'enyi Institute of Mathematics, 1053 Budapest, Re\'altanoda u. 13-15, Hungary.}
\email{nemethi.andras@renyi.mta.hu}



\date{}

\begin{document}

\maketitle


\pagestyle{myheadings} \markboth{{\normalsize
T. L\'aszl\'o and A. N\'emethi}}{ {\normalsize Reduction theorem for lattice cohomology}}

\begin{abstract}
The lattice cohomology of a plumbed 3--manifold $M$ associated with a connected negative definite plumbing graph is an important tool in the study of topological properties of $M$ and in the
comparison of the topological properties with analytic ones, whenever $M$ is realized as complex analytic
singularity link. By definition, its computation is based on the (Riemann--Roch)
weights of the lattice points of
$\Z^s$, where $s$ is the number of vertices of the plumbing graph.
The present article reduces the rank of this lattice to the number of `bad' vertices of the graph.
(Usually the geometry/topology of $M$ is codified exactly by these `bad' vertices
via surgery or other constructions. Their number  measures how far is the plumbing graph from a rational one,
or, how far is $M$ from an $L$--space.)

The effect of the reduction appears also at the level of certain  multivariable
(topological Poincar\'e) series
 as well. Since from these series one can also read the   Seiberg--Witten invariants,
 the Reduction Theorem provides new  formulae for these invariants too.

The reduction also implies the vanishing $\bH^q=0$ of the lattice cohomology for $q\geq
\nu$, where $\nu$ is the number of `bad' vertices. (This bound is sharp.)
\end{abstract}

\section{Introduction}\labelpar{s:1}

\subsection{}
Let $M$ be a plumbed 3--manifold given by a connected
negative definite plumbing graph. It is well--known that $M$ can be considered as
the  link of a normal
surface singularity as well. In this article we will assume that $M$ is a rational homology sphere.

The second author  in \cite{OSZINV,Nlat} associated with such an $M$ (and any
fixed $spin^c$--structure $\frs$ of $M$)  a graded $\Z[U]$--module
$\bH^*(M,\frs)$, called the {\it lattice cohomology} of $M$. The construction was strongly influenced
by the {\it Artin--Laufer program} of normal surface singularities
(targeting topological characterization of certain
analytic invariants), cf. \cite{SWI,OSZINV,Nlat},
 and by the work of Ozsv\'ath and Szab\'o on  {\it Heegaard--Floer theory}, especially \cite{OSzP}
(see also their long list of papers in the subject).

The lattice cohomology  is purely combinatorial.  Conjecturally it
contains all the information about the Heegaard--Floer homology of $M$,  cf. \cite{Nlat}.
 (The conjecture  was verified for several families, cf. \cite{OSZINV,NR,OSSz1,OSSz3}.)
Recently  Ozsv\'ath,  Stipsicz and Szab\'o in \cite{OSSz1} established  a
spectral sequence starting form the lattice cohomology and converging to the Heegaard--Floer
homology. Moreover, they considered the relative version (for knots in $M$) as well \cite{OSSz2,OSSz3}.
A different version of the relative lattice cohomology
associated with local plane curve singularities was identified
with the motivic Poincar\'e series of such germs  \cite{GN}.

Furthermore, in \cite{NSW} the second author proved
that the normalized Euler characteristic of the lattice cohomology
(similarly as of the
Heegaard--Floer homology) coincides with the normalized Seiberg--Witten invariant of the link $M$. This provides
a new combinatorial formula for the Seiberg--Witten invariants.

From the analytic point of view,
the ranks of the lattice cohomology modules and their Euler characteristic have subtle connection with
certain analytic invariants of analytic realizations of $M$ as singularity links
 \cite{OSZINV,Nlat,NSW}. At Euler characteristic level,  Nicolaescu and the second author
 predicted  the coincidence of the equivariant geometric
genus with the Seiberg--Witten invariants of the link (under certain restrictions on the singularity type).
This was proposed
as an extension of the Casson Invariant Conjecture of Neumann--Wahl \cite{NW},
formulated for germs with  integral homology sphere links.
The conjectured identities  were verified for important families of singularities, e.g. for
splice quotient singularities \cite{NO,BN}. The connections continue
at cohomology level as well.  For example,
the existence of the nontrivial higher lattice cohomologies explain conceptually
the failure in the pathological cases of the above `Seiberg--Witten invariant conjecture',
see \cite{SWI,SWII,SWIII,NO} and
\cite{LMN} for counterexamples.
For further details the reader is invited to
consult \cite{OSZINV,Nlat}.
\subsection{}
Usually, the explicit  computation of the lattice cohomology is very hard.
A priori, it is based on the computation of the weights of all lattice points (of a certain $\Z^s$)
and on the description of those `regions', where the weights are less than $N$
for any integer $N$.
The rank of the lattice which appears
in the construction is very  `large': it is the number of vertices of the corresponding
plumbing/resolution graph $G$ of $M$. (The weight  is provided by a Riemann--Roch formula.)

In order to decrease the computational complexity and also to establish the conceptual properties
of the lattice cohomology, one develops the theory in two directions. First, one finds (surgery)
exact sequences (proper to any cohomology theory), see e.g. \cite{Nexseq}.
Or, one tries to decrease the rank of the lattice and simplify the
graded cohomological complexes in such a way that the new presentation
contains essentially no superfluous data, focusing
exactly on the geometry of the 3--manifold. This is what we propose in the present article.

The main result is the  {\it Reduction Theorem} (for the precise form see  \ref{red}),
which reduces the rank of the lattice to $\nu$, the
number of `{\it bad}' vertices of the plumbing graph $G$.
(For the definition of `bad' vertices see \ref{ss:INTRNEW} in this introduction or \ref{ss:bv}.)
This number is definitely much smaller than the total number of vertices
(usually it is even smaller than the number of nodes of the graph). It
provides a `filtration' of  negative definite
plumbing graphs/manifolds,
 which measures how far the graph stays from a rational graph. (For more details see
 \ref{ss:INTRNEW}.)

\subsection{}\label{ss:INTRNEW} Let us explain the role of the Reduction Theorem by the following parallelism.
The following problem is very natural and important: for any CW complex $X$ find a
(minimal) sub--complex $K$ such that $K\subset X$ is a homotopy equivalence.
A modified cohomological version is the following. Fix a cohomology theory $H^*$, and let
$X$ as before. Then find a (minimal) sub-complex $i:K\hookrightarrow S$ such that $i^*:H^*(X)
\to H^*(K)$ is an ismorphism. Definitely, this procedure demands the understanding of the intrinsic properties of $X$.

In our case, we consider  the lattice cohomology $\bH^*$ which associates to
any lattice and weight function $(L,w)$ the module $\bH^*(L,w)$. The pair
$(L,w)$ will be  associated with a
plumbed 3--manifold $M$ (constructed from the graph whose intersection lattice is $L$)
and with a fixed $spin^c$--structure of $M$.
Our Reduction Theorem  finds a (minimal and functorial)
weighted sublattice $(\overline{L},\overline{w})$ with the same cohomology. Doing this we necessarily
find the essential geometric properties of the lattice from the point of view of $\bH^*$.
Here is the precise statement.

$\overline{L}$ is the lattice generated by the `bad' vertices.
Their definition is the following. A graph has no bad vertices if it is rational
(cf. \ref{ss:bv}, this property of singularity links can be compared with the property of
being an $L$-space in the sense of Heegaard Floer theory, cf. \cite{OSZINV}).
Otherwise, if one has to decrease the Euler number of (at most) $\nu$ vertices of the graph
to get a rational graph, we say that these vertices are the `bad' ones.
We will write $\overline{L}_{\geq 0}$ for the first quadrant of $\overline{L}$.
For any fixed $spin^c$--structure $\frs$, and for any
lattice point $\vasi\in\overline{L}_{\geq 0}$ we determine a
 very special universal point $x(\vasi)$ in $L$ (cf. \ref{bek:XI})
 and we set $\overline{w}(\vasi):=w(x(\vasi))$. Then the lattice cohomology of the
 pair $(M,\frs)$, $\bH^*(L,w)$, can be recovered by the isomorphism

 \vspace{2mm}

\noindent {\bf Reduction Theorem:} $\bH^*(L,w)=\bH^*(\overline{L}_{\geq 0},\overline{w})$.

\subsection{}\label{ss:INTR3}

We wish to emphasize again that the reduction to `bad' vertices is not just a technical procedure.
Usually, the key information about the structure of the 3--manifold is coded by them.
Let us support  this statement by some examples.

A  star--shaped graph (the plumbing graph of a Seifert 3--manifold)
has at most one bad vertex, namely the central one.
In this case, the sequence $w(x(i))$ ($i\in \Z_{\geq 0}$) can be determined from
the Seifert invariants, and these weights  are closely related
with Pinkham's computation in \cite{P} of the geometric genus
and of the Poincar\'e series of weighted homogeneous singularities
(the natural analytic realizations of Seifert manifolds as singularity links), see e.g. \cite{SWII,OSZINV}
or Example \ref{ex:Seifert} here.
As a consequence, the geometric genus coincides with the normalized Seiberg--Witten invariant
of the link.
In fact, the output of the Reduction Theorem at the level of  series (cf. \ref{in:zeta})
is exactly the Poincar\'e series associated with the analytic $\C^*$--action.

Another example: let $K$ be the connected sum of $\nu$ irreducible algebraic knots $\{K_i\}_{i=1}^\nu$
of $ S^3$. Consider the surgery 3--manifold $M=S^3_{-d}(K)$ ($d\in \Z_{> 0}$).
Then the minimal number of bad vertices is exactly $\nu$, and the weights $w(x(\vasi))$
are determined from the
semigroups of the knot components $K_i$ (see e.g.  \cite{NR},
where the Reduction Theorem  was already applied).

Even the `naive  case of all nodes' has strong consequences in certain  situations.
(If the graph is minimal good, then decreasing  the Euler numbers of all the nodes we get a minimal rational
graph, hence the set of nodes can be regarded as a set of bad vertices.)
Now, if  we consider the graph/link of a hypersurface singularity
with non-degenerate Newton principal part, then by toric resolution
the nodes correspond to the faces of the
Newton diagram. Hence, this choice of the bad vertices establishes the connection
with the combinatorics of the source toric object, the Newton diagram.

\subsection{}\label{in:zeta}
 The effects of the reduction appear not only at the level of the cohomology modules.
 The lattice cohomology has subtle connections with a certain  multivariable Poincar\'e series
 (defined combinatorially from the graph, which resonates and sometimes equals
the multivariable Poincar\'e series associated with the divisorial filtration indexed by all the divisors
in the resolution, provided by certain analytic realizations) \cite{NPS,NSW,NCL}.
For example, the Seiberg--Witten invariant appears as the `periodic constant' of this series
\cite{NSW,BN,NO}.
(We review these facts in Section~\ref{s:5}.)
The number of variables of this series is again the number of vertices of the plumbing graph.
One of the applications of the Reduction Theorem (and its proof) is that if we eliminate
all the variables except those  corresponding to the `bad' vertices, the new reduced series
still contains all the information about the Seiberg--Witten invariants, see Theorem \ref{thm:redZ}.

The reduction recovers
the vanishing of the reduced lattice cohomology for rational graphs,
proved in \cite[\S 4]{Nlat}. (This corresponds to the
$L$--space property of the link.) More generally, it
implies the vanishing
 $\bH^q(M)=0$ whenever $q\geq \nu$. An alternative
proof of this fact can be found in \cite{Nexseq}, based on
surgery exact sequences. This vanishing is sharp.
 Consider e.g.  the connected sum $K$ of $\nu$ copies of the $(2,3)$--torus knot,
and take the $(-d)$-surgery of the $3$--sphere $S^3$ along $K$, for some $d \in \Z_{>0}$. Then
the minimal number of bad vertices is $\nu$, and
 $\bH^{\nu-1}(S^3_{-d}(K))=\Z$ \cite{NR}.

The second author in \cite{OSZINV,Ng}
associated with $(L,w)$ a set of  graded roots as well (as a refinement of the
0--th order lattice cohomology $\bH^0(L,w)$). Without saying anything more about them,
we note that the proof of the Reduction Theorem guarantees that under the
reduction procedure the roots stay stable as well.

\subsection{}
The organization of the note is the following: Section~\ref{s:3} contains some generalities about the plumbing
graphs and reviews the construction and different interpretations of the lattice cohomology. The next section
defines the `special' cycles $x({\bf i})$, the family of `bad' vertices and provides several technical preliminary results
 about the {\it generalized
Laufer computation sequences}.
(For the original sequences introduced by Laufer, see \cite{Laufer72,Laufer77}.
The present generalizations have their origin
in \cite{OSZINV,Nlat}, where the case of  `almost rational graphs' was treated, i.e.  the $\nu$=1 case.)
At the end of this section we state the Reduction Theorem \ref{red}.
The proof is given in Section 4. It starts with several simplification steps.
The `original' and `reduced' cohomology groups are compared by a projection,
and the isomorphism is guaranteed
by a  Leray type argument, namely   by the fact,  that all the fiber of the projection are
non--empty and contractible. Even the proof of the non--emptiness is  rather hard. The contraction
is done in several steps, and is guided by  high generalizations of properties of computation sequences.
Section 5 contains the corresponding consequences regarding the Poincar\'e series and their
connection with the Seiberg--Witten invariants.

The last section contains a concrete explicit example.

\section{Review of the lattice cohomology}\labelpar{s:3}

\subsection{Generalities about plumbing graphs}\labelpar{be:gene}
We consider  a connected negative definite plumbing graph $G$. It  can be realized as
the resolution graph of some  normal surface singularity $(X,0)$, and the link $M$ of
$(X,0)$ can be considered as the
plumbed 3--manifold associated with $G$. In the sequel we assume that $M$ is a {\it
rational homology sphere},
or, equivalently,  $G$ is a tree and all the genus decorations are zero.
For more details regarding this section, see e.g. \cite{OSZINV,Nsur,Nsurrat,Nlat}.

Let $\widetilde{X}$ be the smooth 4--manifold with boundary $M$ obtained either by
plumbing disc bundles along $G$, or via a  resolution $\pi:\widetilde{X}\to X$ of
$(X,0)$ with resolution graph $G$. Then  $L=H_2(\widetilde{X},\Z)$ is generated by
$\{E_j\}_{j\in {\mathcal J}}$, the cores of the plumbing construction (or the
irreducible components of the exceptional divisor $E:=\pi^{-1}(0)$ of $\pi$).
$L$ is a lattice via the negative definite intersection form $\frI:=\{(E_j,E_i)\}_{j,i}$.
Let $L'$ be the dual lattice $\{l'\in L\otimes \Q\,:\, (l',L)\subseteq  \Z
\}$.  $L'$ is generated by the (anti)dual elements $E_j^*$ defined via
$(E_j^*,E_k)=-\delta_{jk}$ (the negative of the Kronecker symbol). Set $H:=L'/L$. Then $H_1(M,\Z)=H$.
Clearly, the
$E^*_j$ are the columns of $-\frI^{-1}$, and is  known
that
\begin{equation}\label{eq:POS}
\mbox{all the entries of $E^*_j$ are strict positive.}
\end{equation}

If  $l'_k=\sum_jl'_{kj}E_j$ for $k=1,2$, then we write
$\min\{l'_1,l'_2\}:= \sum_j\min\{l'_{1j},l'_{2j}\}E_j$, and  $l'_1\leq l'_2$ if
$l'_{1j}\leq l'_{2j}$ for all $j\in {\mathcal J}$.  Furthermore,
if $l'=\sum_jl'_jE_j$ then we set $|l'|:=\{j\in {\mathcal J}\,:\, l'_j\not=0\}$ for the
{\em support} of $l'$.

The set of {\it characteristic elements} are defined as
$$Char:=\{k\in L': \, (k,x)+(x,x)\in 2\Z \ \mbox{for any $x\in L$}\}.$$
The unique rational cycle $\K\in L'$ which satisfies the system of
{\it adjunction relations} $(\K,E_j)=-(E_j,E_ j)- 2$ for all $j$
is called the  {\em canonical cycle}. Then $Char=\K+2L'$. There is
a natural action of $L$ on $Char$ given by $l*k:=k+2l$; its orbits are
of type $k+2L$. Obviously, $H$ acts freely and transitively on the
set of orbits by $[l']*(k+2L):=k+2l'+2L$.

The first Chern class  realizes an identification
between the $spin^c$--structures $Spin^c(\widetilde{X})$ on
$\widetilde{X}$ and $Char\subseteq L'$.
$Spin^c(\widetilde{X})$ is an $L'$ torsor compatible with the
above action of $L'$ on $Char$.

All the  $spin^c$--structures on
$M$ are obtained by the  restriction $Spin^c(\widetilde{X})\to Spin^c(M)$,
$Spin^c(M)$ is an $H$ torsor, and the actions are compatible with the factorization
$L'\to H$. Hence, one has
an identification of $Spin^c(M)$  with the set of $L$--orbits of
$Char$, and this identification is compatible with the
action of $H$ on both sets. In this way,   any
$spin^c$--structure of $M$ will be represented by an orbit
$[k]:=k+2L\subseteq Char$ (see \cite{GS}).

The \emph{canonical} $spin^c$--structure
corresponds to $[-\K]$.

\subsection{$\Z[U]$--modules.}
\labelpar{ss:3.1} The lattice cohomology has a  graded $\Z[U]$--module structure.
One of its building blocks, $\calt_r^+$, is defined as follows,
 cf. \cite{OSzP,OSZINV}.

Consider the graded $\Z[U]$--module $\Z[U,U^{-1}]$, and   denote by
 $\calt_0^+$ its quotient by the submodule  $U\cdot \Z[U]$.
Its grading is given by  $\deg(U^{-d})=2d$ ($d\geq 0$).
Next, for any graded $\Z[U]$--module $P$ with
$d$--homogeneous elements $P_d$, and  for any  $r\in\Q$,   we
denote by $P[r]$ the same module graded (by $\Q$) in such a way
that $P[r]_{d+r}=P_{d}$. Then set $\calt^+_r:=\calt^+_0[r]$.

\subsection{Lattice cohomology associated with $\Z^s$ and a system
of weights \cite{Nlat}.}\labelpar{ss:lw}
 We fix  a free $\Z$--module, with a {\it fixed basis}
$\{E_j\}_{j=1}^s$, denoted by $\Z^s$. It is also convenient to fix
a total ordering of the index set ${\mathcal J}$, which in the
sequel will be denoted by $\{1,\ldots,s\}$.
Using  the pair $(\Z^s, \{E_j\}_j)$ and a system of weights, in the next paragraphs
 we determine a
cochain complex whose cohomology is our central object.

\bekezdes\labelpar{complex} {\bf The cochain complex.}
$\Z^s\otimes \R$ has a natural cellular decomposition into cubes. The
set of zero--dimensional cubes is provided  by the lattice points
$\Z^s$. Any $l\in \Z^s$ and subset $I\subseteq {\mathcal J}$ of
cardinality $q$  define a $q$--dimensional cube,  denoted by
$(l,I)$ (or only by $\square_q$) which has its
vertices in the lattice points $(l+\sum_{j\in I'}E_j)_{I'}$, where
$I'$ runs over all subsets of $I$. On each such cube we fix an
orientation. This can be determined, e.g.,  by the order
$(E_{j_1},\ldots, E_{j_q})$, where $j_1<\cdots < j_q$, of the
involved base elements $\{E_j\}_{j\in I}$. The set of oriented
$q$--dimensional cubes defined in this way is denoted by $\calQ_q$
($0\leq q\leq s$).

Let $\calC_q$ be the free $\Z$--module generated by oriented cubes
$\square_q\in\calQ_q$. Clearly, for each $\square_q\in \calQ_q$,
the oriented boundary $\partial \square_q$ has the form
$\sum_k\varepsilon_k \, \square_{q-1}^k$ for some
$\varepsilon_k\in \{-1,+1\}$, where the $(q-1)$--cubes $\{\square_{q-1}^k\}_k$
are the  {\em oriented faces} of $\square_q$.
Clearly $\partial\circ\partial=0$, and  the
homology of the chain complex $(\calC_*,\partial)$ is  just the homology of $\R^s$.
A more interesting (co)homology is obtained via   a set of {\em weight
functions}.

\bekezdes\labelpar{weight} {\bf Definition.} A set of functions
$w_q:\calQ_q\to \Z$  ($0\leq q\leq s$) is called a {\em set of
compatible weight functions}  if the following hold:

(a) for any integer $k\in\Z$, the set $w_0^{-1}(\,(-\infty,k]\,)$
is finite;

(b) for any $\square_q\in \calQ_q$ and for any of its faces
$\square_{q-1}\in \calQ_{q-1}$ one has $w_q(\square_q)\geq
w_{q-1}(\square_{q-1})$.

\begin{example}\labelpar{ex:ww}
Assume that some $w_0:\calQ_0\to\Z$  satisfies (a). For any $q\geq 1$ set
\begin{equation*}
w_q(\square_q):=\max\{w_0(v) \,:\, \mbox{$v$ is a vertex of \,$\square_q$}\}.
\end{equation*}
Then $\{w_q\}_q$ is a set of compatible weight functions.
\end{example}

In the presence  of a set of compatible weight functions
$\{w_q\}_q$,  one sets $\calF^q:=\Hom_{\Z}(\calC_q,\et^+_0)$.
Then  $\calF^q$ is a $\Z[U]$--module by
$(p*\phi)(\square_q):=p(\phi(\square_q))$ ($p\in \Z[U]$), and
it has a $\Z$--grading: $\phi\in \calF^q$ is
homogeneous of degree $d\in\Z$ if for each $\square_q\in\calQ_q$
with $\phi(\square_q)\not=0$, $\phi(\square_q)$ is a homogeneous
element of $\et^+_0$ of degree $d-2\cdot w(\square_q)$.
 (In the sequel sometimes we will omit the index $q$ of $w_q$.)

Next, one defines $\delta_w:\calF^q\to \calF^{q+1}$. For
this, fix $\phi\in \calF^q$ and we show how $\delta_w\phi$ acts on
a cube $\square_{q+1}\in \calQ_{q+1}$. First write
$\partial\square_{q+1}=\sum_k\varepsilon_k \square ^k_q$, then
set
$$(\delta_w\phi)(\square_{q+1}):=\sum_k\,\varepsilon_k\,
U^{w(\square_{q+1})-w(\square^k_q)}\, \phi(\square^k_q).$$
Then $\delta_w\circ\delta_w=0$, hence
$(\calF^*,\delta_w)$ is a cochain complex.
Moreover, $(\calF^*,\delta_w)$ has an
augmentation.  Set $\m_w:=\min_{l\in \Z^s}w_0(l)$ and
choose  $l_w\in \Z^s$ such that $w_0(l_w)=\m_w$. Then one defines
the $\Z[U]$--linear map
$\epsilon_w:\et^+_{2\m_w}\longrightarrow \calF^0$ such that
$\epsilon_w (U^{-\m_w-n})(l)$ is the class of $U^{-\m_w+w_0(l)-n}$
in $\et^+_0$ for any $n\in \Z_{\geq 0}$. Then,
 $\epsilon_w$ is injective, and
$\delta_w\circ\epsilon_w=0$.

\bekezdes\labelpar{def12}{\bf Definitions.} The homology of the cochain complex
$(\calF^*,\delta_w)$ is called the {\em lattice cohomology} of the
pair $(\R^s,w)$, and it is denoted by $\bH^*(\R^s,w)$. The
homology of the augmented cochain complex
$$0\longrightarrow\et^+_{2\m_w}\stackrel{\epsilon_w}{\longrightarrow}
\calF^0\stackrel{\delta_w}{\longrightarrow}\calF^1
\stackrel{\delta_w}{\longrightarrow}\ldots$$ is called the {\em
reduced lattice cohomology} of the pair $(\R^s,w)$, and it is
denoted by $\bH_{red}^*(\R^s,w)$.
For any $q\geq 0$, both $\bH^q$ and $\bH_{red}^q$ admit an induced graded
$\Z[U]$--module structure, and one has  graded
$\Z[U]$--module isomorphisms $\bH^q=\bH^q_{red}$ for $q>0$ and
$\bH^0=\et^+_{2\m_w}\oplus\bH^0_{red}$.


\bekezdes\labelpar{mod} {\bf Modification.} Clearly, instead of
all the cubes of $\R^s$ we can consider only those ones which sit
in $[0,\infty)^s$, or only in the `rectangle'
$R:=[0,T_1]\times\cdots\times [0,T_s]$ (for some $T_i\in \Z_{\geq
0}$). In such a case, we write $\bH^*([0,\infty)^s,w)$ or
$\bH^*(R,w)$ for the corresponding lattice cohomologies.

\subsection{The $\bS^*$--realization}\labelpar{HS} A more geometric realization of the
modules $\bH^*$ is the following.
For each $N\in\Z$, define
$S_N=S_N(w)\subseteq\R^s$ as the union of all the cubes $\square_q$
(of any dimension) with $w(\square_q)\leq N$. Clearly, $S_N=\emptyset$,
whenever $N<\m_w$. For any $q\geq0$, set
\begin{equation*}
\bS^q(\R^s,w):=\oplus_{N\geq \m_w}H^q(S_N,\Z).
\end{equation*}
Then $\bS^q$ is  $2\Z$--graded, the $d=2N$--homogeneous elements
 $\bS^q_d$ consists of $H^q(S_N,\Z)$. Also, $\bS^q$ is a $\Z[U]$--module.
 The $U$--action is given by the restriction map
 $r_{N+1}:H^q(S_{N+1},\Z)\longrightarrow H^q(S_N,\Z)$,
 namely, the $N^{th}$--component of $U*(\{\alpha_N\}_N)$ is $r_{N+1}(\alpha_{N+1})$.
 Moreover, for $q=0$, the fixed
 base--point $l_w\in S_{\m_w}$ provides an augmentation $H^0(S_N,\Z)=\Z\oplus\widetilde{H}^0(S_N,\Z)$,
 hence an augmentation of the graded $\Z[U]$--modules
\begin{equation*}
\bS^0=(\oplus_{N\geq \m_w}\Z)\oplus(\oplus_{N\geq \m_w}\widetilde{H}^0(S_N,\Z))=\et^+_{2\m_w}\oplus\bS^0_{red}.
\end{equation*}

\begin{theorem}\labelpar{th:HS}\cite{Nlat}
There exists a graded $\Z[U]$--module isomorphism, compatible with
the augmentations, between $\bH^*(\R^s,w)$ and $\bS^*(\R^s,w)$.
Similar statement is valid for

\noindent $\bH^*([0,\infty)^s,w)$, or for
$\bH^*(\prod_i[0,T_i],w)$.
\end{theorem}
From now on we  denote both  realizations with the same symbol
$\bH^*$, no matter which  one we  use.

\subsection{The lattice cohomology associated with a plumbing graph.}\label{ss:pl}
Let $G$ be a negative definite
plumbing graph as in \ref{be:gene}. Let $s$ be the number of vertices.
Then we can associate to $L=\Z^s$ the free $\Z$--module $\calC_q$ generated by oriented cubes
$\square_q\in\calQ_q$, as in \ref{complex}.

To any $k\in Char$ we associate  weight functions
$\{w_q\}_q$ as follows.
First, we  define $\chi_k:L\to \Z$ by
$$\chi_k(l)=-(l,l+k)/2;$$
and we also write $\m_k:=\min\, \{\, \chi_k(l)\, :\, l\in L\}$.
Then the weight functions are defined as in \ref{ex:ww} via $w_0:=\chi_k$.
The associated lattice cohomologies   will be denoted
by $\bH^*(G,k)$ and  $\bH^*_{red}(G,k)$.

It is proved in \cite{Nlat} that $\bH^*_{red}(G,k)$
is finitely generated over $\Z$.

\begin{remark}\label{re:depk} Although each $k$ provides a different cohomology
module, there are only $|L'/L|$ essentially different ones. Indeed,
 assume that $[k]=[k']$, hence $k'=k+2l$ for some $l\in L$. Then
\begin{equation}\label{eq:k'k}
\chi_{k'}(x-l)=\chi_k(x)-\chi_k(l) \ \ \ \mbox{ for any $x\in L$}.\end{equation}
 Therefore,
the transformation $x\mapsto x':=x-l$ realizes the following
identification:
$$\bH^*(G,k')=\bH^*(G,k)[-2\chi_k(l)].$$
\end{remark}

\subsection{The distinguished representatives $k_r$}\labelpar{ss:distin}
We fix a $spin^c$--structure $[k]$.
Recall, see \ref{be:gene}, that $[k]$ has  the form $k_{can}+2(l'+L)$ for some $l'\in L'$.
Among all the characteristic elements in $[k]$ we will choose a very special one.
Consider the (Lipman, or anti--nef)
cone $$\calS':=\{l'\in L': \, (l',E_v)\leq0 \text{ for any vertex }v\}.$$
\bekezdes {\bf Definition.}\label{def:KR} \cite[(5.4--5.5)]{OSZINV}
We denote by $l'_{[k]}\in L'$ the unique minimal element of $(l'+L)\cap \calS'$ and we call
$k_r:=k_{can}+2l'_{[k]}$ the {\em distinguished representative} of the class $[k]$.

For example, since the minimal element  of  $L\cap \calS'$ is the  zero cycle,
we get  $l'_{[k_{can}]}=0$, and the distinguished representative in $[k_{can}]$
is  the canonical cycle $k_{can}$ itself. In general, $l'_{[k]}\geq 0$.

The classes $k_r$ generalize the canonical cycle for different $spin^c$--structures.
Their importance will be transparent below,  see also \cite{OSZINV,Nlat,Ng} for different applications.
The next properties are proved in
\cite{Nlat}:

\begin{lemma}\labelpar{propF2}

(a) \ $\bH^*(G,k_r)\cong\bH^*([0,\infty)^s,k_r)$ for any $k_r$.

(b) \ The set $\{\bH^*(G,k_r)\}_{[k_r]}$ is independent on
the plumbing representation $G$ of the 3--manifold $M$, hence it associates a
graded $\Z[U]$--module to
any pair $(M,k_r)$ indexed by  $[k_r]\in\, Spin^c(M)$.
\end{lemma}

\subsection{Notation.} \ In the sequel we denote $\chi_{k_{can}}$ by $\chi_{can}$. 

\section{The lattice reduction}\labelpar{s:4a}

\subsection{Computation sequences} \label{ss:cs}
The goal of the present section is to show that
the lattice cohomology of the lattice $L$
(or any rectangle of it) can be reduced to a considerably smaller rank lattice
with properly chosen weight functions.  In this subsection
we introduce the needed generalizations  and we state the main theorem.

The idea of the Reduction Theorem is present  already in \cite{OSZINV}.

The new lattice of rank $\nu$ will be associated with a set of `bad' vertices, and
the `new  weights' will be determined via the `old weights' of certain distinguished
`universal  cycles' of $L$ (determined by the bad vertices). The construction and main
 properties of these cycles are closely related with
generalized Laufer--type computational sequences of $L$. (For Laufer's original
computational sequences see e.g. \cite{Laufer72,Laufer77}.) In particular, in several
paragraphs we will analyze properties of these sequences and of these universal cycles.

We start with their definition.

\bekezdes\label{bek:XI}
 {\bf The definition of the lattice points $x(i_1,\ldots,i_\nu)$.} Suppose we have a
 family of {\it distinguished}  vertices $\ocalj:=\{j_k\}_{k=1}^\nu\subseteq \calj$
 (usually chosen by a certain  geometric property).
Then split the set of vertices $\calj$ into the disjoint union $\overline{\calj}\sqcup\calj^*$.  Furthermore,
let $\{m_j(x)\}_j$ denote  the coefficients of a  rational cycle  $x$, that is $x=\sum_{j\in\calj}m_j(x)E_j$.

In order to simplify the notation we set  ${\bf i}:=(i_1,\ldots,i_j,\ldots,i_\nu)\in\Z^\nu$;
for any $j\in \ocalj$ we write $1_j\in \Z^\nu$ for the vector with all entries zero except at place
$j$ where it is 1,  and for     any
 $I\subseteq \ocalj$ we define $1_I=\sum _{j\in I}1_j$. Similarly, for any
 $I\subseteq \calj$ set $E_I=\sum_{j\in I}E_j$.

Then the cycles $x({\bf i})=x(i_1,\ldots,i_\nu)$ are defined via the next Proposition.

\begin{proposition}\labelpar{lemF1} Fix $[k]$ and $\ocalj\subseteq \calj$ as above.
For any ${\bf i}\in(\Z_{\geq 0})^\nu$ there exists a unique cycle
$x({\bf i})\in L$ satisfying the following properties:
\vspace{1mm}
\begin{itemize}
\item[(a)] $m_{j}(x({\bf i}))=i_j$ for any distinguished  vertex $j\in\ocalj$;
\vspace{1mm}
\item[(b)] $(x({\bf i})+l'_{[k]},E_j)\leq0$ for every `non--distinguished vertex' $j\in\calj^*$;
\vspace{1mm}
\item[(c)] $x({\bf i})$ is minimal with the two previous properties.

\end{itemize}
Moreover,
\ (i)  $x(0,\ldots ,0)=0$;
\ (ii)   $x({\bf i})\geq 0$; and
\ (iii)  $x({\bf i})+E_{\overline{I}}
\leq x({\bf i}+1_{\overline{I}})$ for any $\overline{I}\subseteq \ocalj$.
\end{proposition}
\begin{proof}
The proof is similar to the proof of  \cite[Lemma 7.6]{OSZINV}, valid
for  $\nu=1$ (or to the existence of the Artin's cycle which corresponds to
  $\nu=0$ and the canonical class).

  First we verify the existence of an element $x\in L$ with (a)--(b). By (the proof of)
\cite[7.3]{OSZINV} there exists
$\widetilde{x}\geq \sum_{j\in\ocalj}E_j$ such that $(\widetilde{x}+l'_{[k]},E_j)\leq 0$
for any $j\in \calj$.  Take some $a\in\Z_{>0}$ sufficiently large so that $(a-1)l'_{[k]}\in L$, and
 $h_j:=m_{j}(a\widetilde{x}+(a-1)l'_{[k]})-i_j\geq 0$ for any $j\in \ocalj$. Since $l'_{[k]}\geq 0$, this is
 possible. Then set $x:=a\widetilde{x}+(a-1)l'_{[k]}-\sum_{j\in \ocalj} h_jE_j$. Clearly $m_{j}(x)=i_j$
 for any $j\in\ocalj$ and $(x+l'_{[k]}, E_i)=a(\widetilde{x}+l'_{[k]},E_i)-\sum_{j\in\ocalj}h_j(E_j,E_i)\leq 0$ for any
 $i\in \calj^*$.

 Next, we verify that there is a unique minimal element with (a)--(b). This follows from the fact that
 if $x_1$ and $x_2$ satisfy (a)--(b), then $x:=\min\{x_1,x_2\}$ does too. Indeed, for any $j\in\calj^*$,
 at least for one index $n\in\{1,2\}$ one has $E_j\not\in|x_n-x|$. Then $(x+l'_{[k]},E_j)=(x_n+l'_{[k]},E_j)-
 (x_n-x,E_j)\leq 0$.

 Finally, we verify (i)--(ii)--(iii).  For (ii)  write $x({\bf i})$ as
 $x_1-x_2$ with $x_1\geq 0$, $x_2\geq 0$, $|x_1|\cap |x_2|=\emptyset$. Fix an index $j\in \calj^*$.
 If $j\not\in |x_1|$ then $(l'_{[k]}-x_2,E_j)\leq (l'_{[k]}-x_2+x_1,E_j)\leq 0$. If
 $j\in |x_1|$ then $(l'_{[k]}-x_2,E_j)\leq (l'_{[k]},E_j)\leq 0$, cf. \ref{def:KR}.
 Moreover, $|x_2|\subset  \calj^*$ implies $(-x_2,E_j)\leq 0$ for any $j\in \ocalj$.
 Hence $l'_{[k]}-x_2\in (l'_{[k]}+L)\cap \calS'$,
 which implies $x_2=0$ by the minimality of $l'_{[k]}$.
 This ends (ii) and shows (i) too.
 For (iii) notice that
 $(x({\bf i}+1_{\overline{I}})+l'_{[k]},E_j)  -(E_{\overline{I}},E_j)\leq 0$ for any $j\in \calj^*$,
 hence the result follows from the minimality property (c) applied for $x({\bf i})$.
 \end{proof}

These cycles  satisfy the following universal property as well.

\begin{lemma}\label{lem:lauf} Fix some ${\bf i}\in(\Z_{\geq 0})^\nu$.
Assume that   $x\in L$ satisfies $m_{j}(x)= m_{j}(x({\bf i}))$ for all $j\in\ocalj$.

If $x\leq x({\bf i})$, then there is a `generalized Laufer computation sequence'
connecting $x$ with $x({\bf i})$. More precisely, one constructs a sequence $\{x_n\}_{n=0}^t$ as follows.
Set $x_0=x$. Assume that $x_n$ is already constructed. If for some $j\in\calj^*$
one has $(x_n+l'_{[k]},E_j)>0$ then take $x_{n+1}=x_n+E_{j(n)}$, where $j(n)$ is such an index. If $x_n$
satisfies \ref{lemF1}(b), then stop and set $t=n$. Then this procedure stops after finite steps and
$x_t$ is exactly $x({\bf i})$.

Moreover, along the
computation sequence  $\chi_{k_r}(x_{n+1})\leq  \chi_{k_r}(x_{n})$ for any $0\leq n<t$.
\end{lemma}
\begin{proof}
We show by induction that $x_n\leq x({\bf i})$ for any $0\leq n\leq t$; then
 the minimality property $(c)$ of $x({\bf i})$ will finish the argument. For $n=0$ this is clear.
Assume it is true for $x_n$. Then we have to verify
 that  $m_{j(n)}(x_{n})<m_{j(n)}(x({\bf i}))$. Suppose that this is not true, that is
 $m_{j(n)}(x({\bf i})-x_n)=0$. Then $(x_n+l'_{[k]},E_{j(n)})=(x({\bf i})+
 l'_{[k]},E_{j(n)})-(x({\bf i})-x_n,E_{j(n)})\leq0$, a contradiction.

Finally, notice
that $(x_n+l'_{[k]},E_{j(n)})>0$ implies   $\chi_{k_r}(x_{n+1})\leq\chi_{k_r}(x_n)$.
\end{proof}

Note that the generalized computation sequence usually is not unique, one can make several choices for $j(n)$
at each step $n$.

If the choice of the  {\it distinguished vertices} $\ocalj$ is guided by some specific geometric feature,
then  the cycles $x({\bf i})$ will inherit  further properties (see next subsection).

\subsection{Graphs with `bad' vertices}\label{ss:bv}
In \cite{Nlat} is proved that the reduced lattice cohomology of a {\it rational graph}
(see the definition  below) is trivial; in particular, the lattice cohomology measures
 how `non--rational'  the graph
 is. Any graph can be transformed into a
rational graph by decreasing the decorations of  the vertices.
Indeed, if all the Euler decorations of a graph $G$ are sufficiently negative
(e.g. $(E_j,E)\leq 0$ for any $j$), then $G$ is rational.
This shows that the condition in Definition \ref{def:rat} below can be realized.

Recall that a normal surface singularity is rational if its geometric genus is zero.
This vanishing property was  characterized combinatorially by Artin in terms of the graph \cite{Artin62}:
\begin{equation}\label{eq:artin}
\mbox{rationality}\ \Longleftrightarrow \ \ \chi_{can}(l)> 0 \ \ \mbox{for any $l>0$, $l\in L$.}
\end{equation}
\begin{definition}\label{def:rat}
A connected negative  definite graph is {\em rational} if it is
the resolution graph of a rational singularity, that is,  if
it satisfies Artin's criterion (\ref{eq:artin}).

We say that a graph {\it has $\nu$ bad vertices} if one can find a subset
 of vertices  $\{j_k\}_{k=1}^\nu$, called {\it bad vertices},  such that replacing their decorations
 $e_j:=(E_j,E_j)$ by some more negative integers
$e'_j\leq e_j$ we get a rational graph, cf.
\cite{OSZINV,Nlat,Nexseq,OSzP}.
\end{definition}

A possible set of bad vertices can be chosen in many different ways,
it is not determined uniquely even if its minimal with this property. In fact,
usually we will work with {\it non necessarily minimal} sets.

%
In the sequel we fix a (non--necessarily minimal) set $\ocalj$ of {\it bad vertices} (hence, by
decreasing their decorations one gets a rational graph).
Next,  we start to list some additional properties
satisfied by the cycles  $x({\bf i})$ associated with $\ocalj$, provided by this extra `badness'
assumption.
The first is an addendum of Lemma \ref{lem:lauf}.

\begin{lemma}\label{lem:laufb} Fix some ${\bf i}\in(\Z_{\geq 0})^\nu$.
Assume that   $x\in L$ satisfies $m_{j}(x)= m_{j}(x({\bf i}))$ for all $j\in\ocalj$. Then
 $\chi_{k_r}(x)\geq \chi_{k_r}(x({\bf i}))$.
\end{lemma}

\begin{proof}
Write $x=x({\bf i})-y_1+y_2$ with $y_1\geq 0$, $y_2\geq 0$, both $y_i$ supported on
$\calj^*$, and $|y_1|\cap|y_2|=\emptyset$. Then $\chi_{k_r}(x)=\chi_{k_r}(x({\bf i})-y_1)
+\chi_{k_r}(y_2)+(y_1,y_2)-(x({\bf i})+l'_{[k]},y_2)$.
Via this identity $\chi_{k_r}(x)
\geq \chi_{k_r}(x({\bf i})-y_1)$.
Indeed, $(y_1,y_2)\geq 0$ by support--argument, $-(x({\bf i})+l'_{[k]},y_2)\geq 0$ by definition of
$x({\bf i})$, and $\chi_{k_r}(y_2)\geq 0$ since $y_2$ is supported on a rational subgraph
(cf. \cite[(6.3)]{OSZINV}). On the other hand, by \ref{lem:lauf},
$\chi_{k_r}(x({\bf i})-y_1)\geq \chi_{k_r}(x({\bf i}))$. \end{proof}

The computation sequence of Lemma \ref{lem:lauf}
is a generalization of Laufer's computation sequence targeting
{\it Artin's fundamental cycle} $z_{min}$, the minimal non--zero cycle of $\calS'\cap L$ \cite{Laufer72}.
In fact, for rational graphs, the algorithm is more precise. For further references we cite it here:

\begin{bekezdes}\labelpar{LC}{\bf Laufer's Criterion of Rationality \cite{Laufer72}.} \
{\it Let $\{z_n\}_{n=0}^T$ be the  computation sequence
(similar as above with $[k]=[k_{can}]$)
connecting $z_0=E_{j}$ (for some $j\in\calj$)
 and the Artin's fundamental cycle $z_T=z_{min}$.
 (This means that $z_{n+1}=z_n+E_{j(n)}$ for some $j(n)$, where
 $(z_n,E_{j(n)})>0$.)
 Then the graph is  {\it rational} if and only if
at every step $0\leq n<T$ one has $(E_{j(n)},z_n)=1$.
 The same statement is true for a sequence connecting $z_0=E_{I}$  with $z_{min}$ for any connected $E_I$.

 (Both statement can be reinterpreted by the identity $\chi_{can}(E_{I})=\chi_{can}(z_{min})=1$.)}
\end{bekezdes}

In some of the applications regarding the cycles $x({\bf i})$
 we do not really need their precise forms,
rather  the values $\chi_{k_r}(x({\bf i}))$. These can be
computed inductively thanks to the following.

\begin{proposition}\labelpar{propF1}
For any  $k_r\in Char$,  ${\bf i}\in(\Z_{\geq 0})^\nu$ and
$j\in \overline{\calj}$ one has
\begin{equation*}
\chi_{k_r}(x({\bf i}+1_j))=
\chi_{k_r}(x({\bf i}))+1-(x({\bf i})+l'_{[k]},E_{j}).
\end{equation*}
Moreover,  $\chi_{k_r}(x(0,\ldots,0))=0$.
\end{proposition}

\begin{proof}
We consider  the  computation sequence $\{x_n\}_{n=0}^t$
connecting $x({\bf i})+E_{j}$ and $x({\bf i}+1_j)$ and
we  prove  that $(x_n+l'_{[k]},E_{j(n)})$ is exactly $1$ for any $0\leq n<t$.
Indeed, we take $z_n:=x_n-x({\bf i})$ for $0\leq n\leq t$ and
 one verifies that $\{z_n\}_{n=0}^t$ is the beginning of a
 Laufer sequence $\{z_n\}_{n=0}^T$ (with $t\leq T$)
connecting $E_{j}$ with $z_{min}$ (as in \ref{LC}).
This follows from $(x_n+l'_{[k]},E_{j(n)})>0$ and $(x({\bf i})+l'_{[k]},E_{j(n)})\leq 0$.
Moreover,
the values $(z_n,E_{j(n)})$ will stay unmodified for every $n$ if we replace our graph $G$
with the rational graph $\widetilde{G}$ by decreasing the decorations of the bad vertices.
 Therefore, by Laufer's Criterion \ref{LC}, $(z_n,E_{j(n)})=1$ in
 $\widetilde{G}$, hence consequently in $G$ too. This shows that
\begin{equation*}
1=(x_n-x({\bf i}),E_{j(n)})=(x_n+l'_{[k]},E_{j(n)})-
(x({\bf i})+l'_{[k]},E_{j(n)})
\geq (x_n+l'_{[k]},E_{j(n)}).
\end{equation*} Since $(x_n+l'_{[k]},E_{j(n)})>0$, this number  must equal 1.

This shows   $\chi_{k_r}(x_{n+1})=\chi_{k_r}(x_n)$, or
$\chi_{k_r}(x({\bf i}+1_j))=\chi_{k_r}(x({\bf i})+E_{j})$.
\end{proof}

The next technical result about computation sequences is crucial in the proof of the main result.

\begin{proposition}\label{prop:G}
Fix ${\bf i}\in(\Z_{\geq 0})^\nu$ and a subset $\overline{J}\subseteq \ocalj$.
Let $\cals\iij\subseteq \calj^*$
be the support of $x({\bf i}+1_{\overline{J}})-x({\bf i})-E_{\overline{J}}$.

(I) \ For any subset $\cals'\subseteq \cals\iij$  one can find  a generalized
Laufer computation sequence $\{x_n\}_{n=0}^t$ as in Lemma \ref{lem:lauf} connecting
$x_0=x({\bf i})+E_{\overline{J}} +E_{ \cals'}$ with $x_t=x({\bf i}+1_{\overline{J}})$ with the property that there exists a certain
$t_s$ ($0\leq t_s\leq t$)  such that

(a) $x_{t_s}=x({\bf i})+E_{\overline{J}}+E_{\cals\iij}$, and

 (b)  $\chi_{k_r}(x_n)=\chi_{k_r}(x({\bf i}+1_{\overline{J}}))$  for any $t_s\leq n\leq t$,
or,   $(x_n+l'_{[k]},E_{j(n)})=1$ for $t_s\leq n<t$.

 (II) \ Let $\widetilde{\cals}$ be a subset of $\calj^*$  such that
 \begin{equation}\label{eq:sss}
 \chi_{k_r}(x({\bf i})+E_{\overline{J}\cup\widetilde{\cals}})=
 \chi_{k_r}(x({\bf i}+1_{\overline{J}})).
 \end{equation}
 Then $\widetilde{\cals}\subseteq \cals\iij$. Moreover, there exists a computation
 sequence $\{x_n\}_{n=0}^t$ as in Lemma \ref{lem:lauf} connecting
$x_0=x({\bf i})+E_{\overline{J}}$ with $x_t=x({\bf i})+E_{\overline{J}\cup\widetilde{\cals}}$
such that  $\chi_{k_r}(x_{n+1})\leq \chi_{k_r}(x_n)$ for any $0\leq n<t$.

 (III) \ For any  cycle $l^*>0$ with support $|l^*|\subseteq \calj^*\setminus \cals\iij$, there exists a
 computation sequence  $\{x_n\}_{n=0}^t$  of type  $x_{n+1}=x_n+E_{j(n)}$ (for  $n<t$),
 $x_0=x({\bf i})+E_{\overline{J}\cup\cals\iij}$ and
 $x_t=x({\bf i})+E_{\overline{J}\cup\cals\iij}+l^*$ such that $\chi_{k_r}(x_{n+1})\geq \chi_{k_r}(x_{n})$
 for any $0\leq n< t$ (that is, with $(x_n+l'_{[k]},E_{j(n)})\leq 1$).
\end{proposition}
\begin{proof}
(I) We will use the following notation: for any $x\geq x({\bf i})+E_{\overline{J}}$ we write
 $\|x\|$ for the support $|x- x({\bf i})-E_{\overline{J}}|$. Note that Lemma \ref{lem:lauf}
guarantees the existence of a  computation sequence connecting $x({\bf i})+E_{\overline{J}\cup \cals'}$ with
$x({\bf i}+1_{\overline{J}})$.  We consider such a sequence $\{x_n\}_{n=0}^t$ constructed in such a way
that in the procedure of choices of $j(n)$'s  at the first steps
we try to  increase  $\|x_n\|$ as much as possible.
 More precisely, for any $0\leq n<t_1$, the index $j(n)\in\calj^*$ is chosen as follows:
 \begin{equation}\label{eq:jn}
 \left\{\begin{array}{l} (x_n+l'_{[k]},E_{j(n)})>0 \\
 E_{j(n)}\not\in \|x_n\|.\end{array}\right.
 \end{equation}
 Assume that this  stops for $n=t_1$, that is,
for $n=t_1$ there is no index $j(n)\in\calj^*$ which would  satisfy (\ref{eq:jn}).
We claim that
$\|x_{t_1}\|=\|x({\bf i}+1_{\overline{J}})\|=\cals\iij$,
hence $t_s=t_1$ satisfies part (a) of the proposition.

Indeed, assume that this is not the case. Then we continue the construction of the sequence, and let
$t_2+1$ be the first index when $\|x\|$ increases again, that is $\|x_n\|=\|x_{t_1}\|$ for
$t_1\leq n\leq t_2$ and $\|x_{t_2+1}\|=\|x_{t_1}\|\cup \{j^*\}\not=\|x_{t_1}\|$ for some $j^*\in\calj^*$.
Hence $j^*=j(t_2)$.

Since $(x_{t_2}+l'_{[k]},E_{j^*})>0$ and $(x_{t_1}+l'_{[k]},E_{j^*})\leq 0$,
we get $(x_{t_2}-x({\bf i}),E_{j^*})>-(x({\bf i})+l'_{[k]},E_{j^*})\geq (x_{t_1}-x({\bf i}),E_{j^*})$.
Since $x_{t_2}-x({\bf i})$ and $x_{t_1}-x({\bf i})$ have the same support, which does not contain $j^*$,
this strict inequality can happen only if  $(x_{t_1}-x({\bf i}),E_{j^*})>0$.
By the same argument, in fact, there exists a connected component $C$ of the reduced cycle $x_{t_1}-x({\bf i})$
such that \begin{equation}\label{eq:GG}
((x_{t_2}-x({\bf i}))|_C,E_{j^*})> (C,E_{j^*})>0.\end{equation}
Next, we analyze the restriction of the sequence $z_n:=x_n-x({\bf i})$ to $C$ for $t_1\leq n\leq t_2$.
First note that $(z_n,E_{j(n)})=(x_n+l'_{[k]},E_{j(n)})-(x({\bf i})+l'_{[k]},E_{j(n)})>0$.
If $E_{j(n)}$ is supported by $C$ then it does not intersect any other components of $x_{t_1}-x({\bf i})$,
hence $(z_n|_C,E_{j(n)})>0$ too.
Let us consider that subsequence $\tilde{z}_*$
of $z_n|_C$ which is obtained from $z_n|_C$ by eliminating those steps from the computation sequence
of $\{x_n\}_{n=t_1}^{t_2}$  which correspond to elements $j(n)$ not supported by $C$.
Then the sequence starts with $E_C$, ends with $(x_{t_2}-x({\bf i}))|_C$, it is the beginning of a Laufer
sequence connecting the {\it connected}
 $E_C$ with the fundamental cycle of $C$, but at the step $t_2$ one has
$(z_{t_2}|_C,E_{j(t_2)})\geq 2$, cf. (\ref{eq:GG}).

Note also that the sequence $z_n|_C$ is reduced along $\overline{J}$, hence along the procedure we do not add any base element from $\overline{J}$, hence if we decrease the self--intersections of these vertices we will not
modify the Laufer data along the sequence. Hence, we can assume that $C$ is  supported by a
rational graph.
But this contradicts the existence of  $\tilde{z}_*$, cf. \ref{LC}.

Part (b) uses the same argument.
We fix a connected component  of $x_{t_s}-x({\bf i})$. Since in the Laufer steps the components do not
interact, we can even assume that the support of $x_{t_s}-x({\bf i})$ is connected.
 Then  $x_n-x({\bf i})$ for  $n\geq t_s$ is part of the computations sequence connecting
 the reduced connected $x_{t_s}-x({\bf i})$ to its fundamental cycle. Since we may assume that
  $C$ is rational (since the steps do not involve $\overline{J}$),
  along the sequence we must have
 $(x_n-x({\bf i}),E_{j(n)})=1$  by \ref{LC}. This happens only if
  $(x_n+l'_{[k]},E_{j(n)})=1$ and $(x({\bf i})+l'_{[k]},E_{j(n)})=0$.

(II) Assume that $\widetilde{\cals}\not\subseteq \cals\iij$, and set
$\cals':=  \widetilde{\cals}\cap\cals\iij$
and  $\Delta\cals:=  \widetilde{\cals}\setminus \cals\iij$.
Take  a computation sequence $\{x_n\}_{n=0}^t$ as in (I) connecting
$x({\bf i})+E_{\overline{J}\cup \cals'}$ with $x({\bf i}+1_{\overline{J}})$.
 Since $\chi_{k_r}(x_n)$ is non--increasing, cf.  \ref{lem:lauf},
$1-(E_{j(n)}, x_n+l'_{[k]})\leq 0$.
Therefore,
$1-(E_{j(n)}, x_n+E_{\Delta\cals}+l'_{[k]})\leq 0$
too, since $j(n)\not\in \Delta\cals$. Since $\{x_n+E_{\Delta\cals}\}_n$ connects
$x({\bf i})+E_{\overline{J}\cup \widetilde{\cals}}$ with $x({\bf i}+1_{\overline{J}})+E_{\Delta\cals}$, we get
$$\chi_{k_r}(x({\bf i})+E_{\overline{J}\cup \widetilde{\cals}})\geq \chi_{k_r}
(x({\bf i}+1_{\overline{J}})+E_{\Delta\cals}).$$
This together  with assumption (\ref{eq:sss}) and Lemma \ref{lem:laufb} guarantee that, in fact,
\begin{equation}\label{eq:EQ}
\chi_{k_r}(x({\bf i}+1_{\overline{J}})+E_{\Delta\cals})=
\chi_{k_r}(x({\bf i}+1_{\overline{J}})).\end{equation}
On the other hand,
\begin{equation*}
\chi_{k_r}(x({\bf i}+1_{\overline{J}})+E_{\Delta\cals})-
\chi_{k_r}(x({\bf i}+1_{\overline{J}}))
=\chi_{can}(E_{\Delta\cals})-(E_{\Delta\cals},
x({\bf i}+1_{\overline{J}})+l'_{[k]})\geq \chi_{can}(E_{\Delta\cals}),
\end{equation*}
where the last inequality follows from the definition of $x({\bf i}+1_{\overline{J}})$. Since
$\chi_{can}(E_{\Delta\cals})$ is the number of connected components of $E_{\Delta\cals}$, it is strictly
positive, a fact which contradicts (\ref{eq:EQ}).

For the second part we construct a computation sequence as in (I), applied for
$\cals'=0$, in such a way that first we choose only the $j(n)$'s from $\widetilde{\cals}$. We claim that in this
way we fill in all $\widetilde{\cals}$. Indeed, assume that this procedure stops at the level of $x_m$; that is,
$x({\bf i})+E_{\overline{J}}\leq x_m <  x({\bf i})+E_{\overline{J}\cup \widetilde{\cals}}$
and
\begin{equation}\label{eq:ineqtilde}
(E_j,x_m+l'_{[k]})\leq 0 \ \ \mbox{for all $j\in \Delta\widetilde{\cals}:=
\widetilde{\cals}\setminus ||x_m||$}.
\end{equation}
Then
$$\chi_{k_r}( x({\bf i})+E_{\overline{J}\cup \widetilde{\cals}})
-\chi_{k_r}(x_m)=\chi_{can}(E_{\Delta\widetilde{\cals}})-
(E_{\Delta\widetilde{\cals}}, x_m+l'_{[k]})\geq
\chi_{can}(E_{\Delta\widetilde{\cals}}),$$
where the last inequality follows from (\ref{eq:ineqtilde}). Since
$\chi_{can}(E_{\Delta\widetilde{\cals}})>0$,   the
assumption (\ref{eq:sss}) imply
 $\chi_{k_r}(x_m)<\chi_{k_r}( x({\bf i}
+1_{\overline{J}}))$, a fact which contradicts Lemma \ref{lem:laufb}.

(III) The statement follows by induction from the following fact: if $l^*>0$,
$|l^*|\subseteq\calj^*\setminus \cals\iij$, then there exists $j\in|l^*|$ so that
$$\chi_{k_r}(x({\bf i})+E_{\overline{J}\cup\cals\iij}+l^*-E_j)\leq
\chi_{k_r}(x({\bf i})+E_{\overline{J}\cup\cals\iij}+l^*).$$
Indeed, if not, then $(E_j,   x({\bf i})+E_{\overline{J}\cup\cals\iij}+l'_{[k]}+l^*-E_j)\geq 2$
for any $j\in |l^*|$. On the other hand,
$(E_j,  x({\bf i})+E_{\overline{J}\cup\cals\iij}  +l'_{[k]})\leq 0$, by the proof of part (I) (namely, the choice of
$t_s=t_1$), or by the definition of $\cals\iij$. Therefore,
$(E_j,l^*-E_j)\geq 2$, or, $(E_j,l^*+k_{can})\geq 0$ for all $j$. Summing up over the coefficients of $l^*$,
we get $(l^*,l^*+k_{can})\geq 0$, which contradicts (\ref{eq:artin}) since the subgraph generated by $|l^*|$
is rational.
\end{proof}

\subsection{The lattice reduction.}\labelpar{ss:reduction}
Now we are ready to formulate the main result of this section: in the definition of
the lattice cohomology we wish to replace the (cubes of the) lattice $L$ with cubes of a smaller rank
free $\Z$--module associated with the bad vertices.

\bekezdes\label{bek:331}
 {\bf Definition of the (quadrant of the) new free $\Z$--module.}   Let us fix $[k]$ and
and a set of $\mu$ bad vertices. We define  $\overline{L}=(\Z_{\geq 0})^\nu$ and the function
$\overline{w}_0:(\Z_{\geq 0})^\nu\to \Z$ by
\begin{equation}
\overline{w}_0(i_1,\ldots,i_\nu):=\chi_{k_r}(x(i_1,\ldots,i_\nu)).
\end{equation}
 Then $\overline{w}_0$ defines a set $\{\overline{w}_q\}_{q=0}^\nu$ of compatible weight functions
 depending on $[k]$, defined similarly as in \ref{ex:ww}, denoted by $\overline{w}[k]$.

\begin{theorem}{\bf (Reduction Theorem)}\labelpar{red} \
Let $G$ be a negative definite connected graph and let $k_r$ be the distinguished representative
of a characteristic class. Suppose $\overline{\calj}=\{j_k\}_{k=1}^\nu$ is a (non--necessarily minimal)
set of bad vertices,  and $(\overline{L},\overline{w}[k])$ is the
first quadrant of the new weighted free $\Z$--module
associated with $\overline{\calj}$ and $k_r$. Then there is a graded $\Z[U]$--module isomorphism
\begin{equation}\label{eq:reda}
\bH^*(G,k_r)\cong\bH^*(\overline{L},\overline{w}[k]).
\end{equation}
\end{theorem}
Note that via Lemma \ref{propF2}, (\ref{eq:reda}) is equivalent to the isomorphism:
\begin{equation}\label{eq:redb}
\bH^*([0,\infty)^s,k_r)\cong\bH^*([0,\infty)^\nu,\overline{w}[k]).
\end{equation}
%

\begin{corollary}\labelpar{corF1} Fix $\nu\geq 1$.
If a graph $G$ has $\nu$ bad vertices then $\bH^q(G,k)=0$ for any $q\geq \nu$ and $k\in Char$.
\end{corollary}
\begin{proof}
Theorems \ref{red} and \ref{th:HS} provide an
 isomorphism $\bH^*(G)=\oplus_{N}H^*(\overline{S}_N,\Z)$.
But $\overline{S}_N$ is a compact cubical subcomplex of $\R^\nu$, hence the vanishing follows.
\end{proof}
The statement of Corollary \ref{corF1}
was proved in  \cite{OSZINV} for $\nu = 1$, and in general in
\cite{Nexseq}  using surgery exact sequences of lattice cohomology.

\section{The proof of Reduction Theorem}

\subsection{Notations, assumption.} \label{ss:assumption}
In this section we abbreviate $k_r$ into $k$, $\overline{w}[k]$ into $\overline {w}$.

Assume that there exists a pair  $j, j'\in\ocalj$, $j\not=j'$, such that $(E_j,E_{j'})=1$. Then we can blow up
the intersection point $E_j\cap E_{j'}$. We have to observe two facts. First,
the lattice cohomology $\bH^*(G,k)$ is stable with respect to this blow up \cite{Nlat,Nexseq}. Second,
the `strict transform' of the set $\ocalj$ can serve as a new set of bad vertices
and the right hand side of (\ref{eq:reda}) stays stable as well. Therefore, by additional blow ups,
{\it  we can assume that}
$$(E_j, E_{j'})=0 \ \ \mbox{for every pair  $j, j'\in\ocalj$, $j\not=j'$.} $$

\subsection{The first step. Comparing $S_N$ and $\overline{S}_N$.}\label{Leray}

We consider the projections  $\phi:(\Z_{\geq 0})^s\to (\Z_{\geq 0})^\nu$
and  $\phi:[0,\infty)^s\to [0,\infty)^\nu$
given by $(m_j)_{j\in\calj}\mapsto (m_{j})_{j\in\ocalj}$. This induces a projection
of the cubes too. If $(l,I)\in \calQ(L)$ is a cube of $L$, then write $I$ as $\overline{I}\cup I^*$,
where
$\overline{I}=I\cap \ocalj$ and $I^*=I\cap \calj^*$. Then the vertices of $(l,I)$ are projected via $\phi$ into the
vertices of the cube $(\phi(l),\overline{I})\in \calQ(\overline{L})$ of $\overline{L}$.
It is convenient to write $\overline {I}:=\phi(I)$ and $\phi(l,I):=(\phi(l),\overline{I})$.

By \ref{lem:laufb}, we get that  for any $l\in (\Z_{\geq 0})^s$ we have $w(l)\geq \overline{w}(\phi(l))$, hence
\begin{equation}\label{eq:ineq}
w((l,I))\geq \overline{w}(\phi(l,I)) \ \  \ \mbox{for any cube $(l,I)\in \calQ(L)$}.
\end{equation}
Recall that for any $N$ we define $S_N\subseteq [0,\infty)^s$ as the union of cubes of
$[0,\infty)^s$ of weight $\leq N$.
Similarly, let $\overline{S}_N\subseteq [0,\infty)^\nu$ be the
union of cubes $({\bf i},\overline{I})$
 with $\overline{w}({\bf i},\overline{I})\leq N$.
 Then, the statement of Theorem \ref{red}, via Theorem
 \ref{th:HS},  is equivalent to the fact that
 \begin{equation}\label{eq:leray}
\mbox{ {\it
$S_N$ and $\overline{S}_N$ have the same cohomology groups for any integer $N$.}}
\end{equation}
%
%
Note that by (\ref{eq:ineq}) $\phi(S_N)\subseteq \overline{S}_N$, and by construction
 $\phi|_{S_N}:S_N\to \overline{S}_N$ is a cubical map.
For any $\ii\subseteq \overline{S}_N$
we consider $\phi^*_N\ii\subseteq S_N$ defined as
the union of all cubes $(l,I)\subseteq S_N$ with  $\phi(l,I)=\ii$.
[We warm the reader that this is not the inverse image
$(\phi|_{S_N})^{-1}\ii$, rather it is the closure of the inverse image of the interiour
of the cube $\ii$; see also below.] If $\psi:[0,\infty)^s\to [0,\infty)^{s-\nu}$ is the second projection
on the $\calj^*$--coordinate direction, then $\phi^*_N\ii$ is the product of $\psi(\phi^*_N\ii)$ with the
cube $\ii$; in particular, it has the homotopy type of $\psi(\phi^*_N\ii)$.

A Mayer--Vietoris inductive (or Leray type spectral sequence)
argument shows that (\ref{eq:leray}) follows from
\begin{equation}\label{eq:leray3}
\mbox{ {\it
$\phi^{*}_N\ii$ is non--empty and contractible  for any $\ii\in\overline{S}_N$.}}
\end{equation}

\subsection{Generalities about contractions.}\label{ss:2}
 {\it In the sequel we fix a cube $\ii$ from $\overline{S}_N$}  and we start
to prove (\ref{eq:leray3}).  For any such cube $\ii$ we also consider the inverse image
$\phi^{-1}\ii$ consisting of the union of
all cubes $(l,I)$ of $[0,\infty)^s$ with $\phi(l,I)\subseteq \ii$ (not
necessarily from  $S_N$). We can also consider $(\phi|_{S_N})^{-1}\ii$,
the union of cubes $(l,I)$ from $S_N$ with $\phi(l,I)\subseteq \ii$.
Clearly, $$\phi^*_N\ii\subseteq (\phi|_{S_N})^{-1}\ii\subseteq \phi^{-1}\ii.$$
Note that $\phi^{-1}\ii$ is the product of the cube   $\ii$ with $[0,\infty)^{s-\nu}$.
Our goal is to contract this `fiber direction  space' $[0,\infty)^{s-\nu}$
in such a way that  along the contraction $\chi_{k}$ does not increase, and the
contraction preserves the subspaces $\phi^*_N\ii$ and $(\phi|_{S_N})^{-1}\ii$ as well.

The cycles supported on $\calj^*$ (`fiber direction') will be denoted by  $l^*=\sum_{j\in\calj^*}m_jE_j$.
For any pair $l_1^*$ and $l_2^*$ with
$l_1^*\leq l_2^*$ we consider the real $s$--dimensional rectangle
 $R_{\ii}(l_1^*,l_2^*)$, the product of a rectangle in the $(s-\nu)$--dimensional space
with the cube $\ii$:  it is the convex closure of the lattice points, which have the form
$$x({\bf i})+E_{\overline{J}}+l^* \ \ \mbox{with} \ \ \ \overline{J}\subseteq \overline{I} \ \ \
\mbox{and} \ \ l^*\in  L, \ \ l_1^*\leq l^*\leq l_2^*.$$
We extend this notation allowing $l_2^*$ to have all its entries  $\infty$.

Note that the lattice points $x({\bf i})+E_{\overline{J}}+l^*$, being in $[0,\infty)^s$, are effective,
hence the relevant $l^*$ satisfies $l^*\geq l^*_{1,min}:=-x({\bf i})+\sum_{j\in\ocalj}i_jE_j$
(the projection of $-x({\bf i})$ on the $\calj^*$-components). In particular,
 $R_{\ii}(l^*_{1,min},\infty)=\phi^{-1}\ii\subseteq [0,\infty)^s$, and  we can assume
 that $l_1^*$ and $l_2^*$ satisfy $l^*_{1,min}\leq l_1^*\leq l_2^*\leq \infty$. Note also that
 $l_{1,min}^*\leq 0$.

\vspace{2mm}

We start to discuss the existence of  a contraction
$c:R_{\ii}(l_1^*,l_2^*+E_j)\to R_{\ii}(l_1^*,l_2^*)$ for some $j\in\calj^*$, acting in the direction of the
$\calj^*$--coordinates and having the property that
 $\chi_k$ will  not increase along it. The map
$c$ is defined  as follows. If a lattice point $l$ is in $R_{\ii}(l_1^*,l_2^*)$, then
 $c(l)=l$. Otherwise $l$ has the form $l=x({\bf i})+E_{\overline{J}}+l^*+E_j$ for some
 $l^*$ with $l_1^*\leq l^*\leq l_2^*$ and $m_j(l^*)=m_j(l^*_2)$. Then set $c(l)=l-E_j$.
The next criterion guarantees that $\chi_k$ does not increase along this contraction.

\begin{lemma}\label{lem:con1}
Assume that for some $l_2^*$ and $j\in\calj^*$ one has
$$\chi_k(x({\bf i})+E_{\overline{I}}+l^*_2+E_j)\geq \chi_k(x({\bf i})+E_{\overline{I}}+l^*_2).
$$
Then, for any $l^*$ with $l_1^*\leq l^*\leq l_2^*$ and $m_j(l^*)=m_j(l^*_2)$,
and for every $\overline{J}\subseteq \overline{I}$, one also has
$$\chi_k(x({\bf i})+E_{\overline{J}}+l^*+E_j)\geq \chi_k(x({\bf i})+E_{\overline{J}}+l^*).
$$
Therefore, $\chi_k(c(l))\leq \chi_k(l)$ for any $l\in  R_{\ii}(l_1^*,l_2^*+E_j)$.
\end{lemma}
\begin{proof} Use $\chi_k(z+E_j)=\chi_k(z)+1-(E_j,z+l'_{[k]})$ and
$(E_j, E_{\overline{I}}-E_{\overline{J}}+l^*_2-l^*)\geq  0$.\end{proof}
The following lemma generalizes  results of \cite[\S\,3.2]{Nlat}, where the case $\nu=1$ is treated.
\begin{lemma}\label{lem:con2}
Assume that for some fixed $l_2^*$ there exists an infinite sequence of cycles
$\{x^*_n\}_{n\geq 0}$, $x^*_n=\sum_{j\in\calj^*}m_{j,n}E_j$,  with $x^*_0=l_2^*$ such that

\begin{itemize}
\item[(a)] $x^*_{n+1}=x^*_n+E_{j(n)}$ for some $j(n)\in\calj^*$, $n\geq 0$;
\item[(b)] $\chi_k(x({\bf i})+E_{\overline{I}}+x^*_{n+1})\geq \chi_k(x({\bf i})+E_{\overline{I}}+x^*_n)$
for any $n\geq 0$.
\item[(c)] for any fixed $j$ the sequence $m_{j,n}$ tends to infinity as $n$ tends to infinity;
\end{itemize}
Then there exists a contraction of  $R_{\ii}(l_1^*,\infty)$ to  $R_{\ii}(l_1^*, l_2^*)$
along which $\chi_k$ is non--increasing.
\end{lemma}
\begin{proof} Use Lemma \ref{lem:con1} and induction over $n$. 
\end{proof}
Symmetrically, by similar proof, one has the following statements too.

\begin{lemma}\label{lem:con3} \

(I) \ For any  fixed $l_1^*$ and $j\in\calj^*$ with $l_1^*-E_j\geq l_{1,min}^*$ if
$$\chi_k(x({\bf i})+l^*_1-E_j)\geq \chi_k(x({\bf i})+l^*_1),
$$
then for any $l^*$ with $l_1^*\leq l^*\leq l_2^*$ and $m_j(l^*)=m_j(l^*_1)$,
and for every $\overline{J}\subseteq \overline{I}$, one also has
$$\chi_k(x({\bf i})+E_{\overline{J}}+l^*-E_j)\geq \chi_k(x({\bf i})+E_{\overline{J}}+l^*).
$$
Therefore, $ R_{\ii}(l_1^*-E_j,l_2^*)$ contracts onto $ R_{\ii}(l_1^*,l_2^*)$ such that
$\chi_k$ does non increase  along the contraction.

(II) \ Assume that there exists a sequence of cycles
$\{x^*_n\}_{n=0}^t$ with $x^*_0=l_{1,min}^*$  and  $x^*_t=l_{1}^*$ such that for any $0\leq n<t$
one has
\begin{itemize}
\item[(a)] $x^*_{n+1}=x^*_n+E_{j(n)}$ for some $j(n)\in\calj^*$,
\item[(b)] $\chi_k(x({\bf i})+x^*_{n})\geq \chi_k(x({\bf i})+x^*_{n+1})$.
\end{itemize}
Then there exists a contraction of  $R_{\ii}(l_{1,min}^*,l_2^*)$ to  $R_{\ii}(l_1^*, l_2^*)$
along which $\chi_k$ is non--increasing.
\end{lemma}

\subsection{Contractions.}\label{ss:3}
In this subsection we apply the results of the previous subsection \ref{ss:2} in order to contract the triple
$(\phi^{-1}\ii,(\phi|_{S_N})^{-1}\ii,\phi^*_N\ii)$.

First we show the existence of a sequence of cycles
$\{x^*_n\}_{n=0}^t$ with $x^*_0=l_{1,min}^*$ and $x^*_t=0$  which satisfies the assumptions
of Lemma \ref{lem:con3}(II). This follows inductively from the following lemma.
\begin{lemma}\label{lem:con4} For any $x^*$  with $l^*_{1,min}\leq x^*<0$ and supported on $\calj^*$
there exists at least one index
$j\in |x^*|$ such that
\begin{equation}\label{eq:xx}
\chi_k(x({\bf i})+x^*)\geq \chi_k(x({\bf i})+x^*+E_j).\end{equation}
\end{lemma}
\begin{proof}
(\ref{eq:xx}) is equivalent to $(E_j,x({\bf i})+l'_{[k]}+x^*)\geq 1$ for some $j\in |x^*|$.
Assume the opposite, that is,  $(E_j,x({\bf i})+l'_{[k]}+x^*)\leq 0$ for every $j\in|x^*|$.
 On the other hand, for  $j\in\calj^*\setminus |x^*|$ one has $(E_j,x^*)\leq 0$ and
 $(E_j,x({\bf i})+l'_{[k]})\leq 0$ by \ref{lemF1}(b). Hence  $(E_j,x({\bf i})+l'_{[k]}+x^*)\leq 0$
 for every $j\in\calj^*$.
 This contradicts the minimality of $x({\bf i})$ in \ref{lemF1}(c).
 \end{proof}

 In particular, Lemma  \ref{lem:con3}(II) applies for $l_1^*=0$ and any $l_2^*\geq 0$ (including $\infty$).

Next, we search for a convenient small cycle $l_2^*$ for which Lemma \ref{lem:con2} applies as well.
First we show that $l^*_2=\infty$ can be replaced by
$x({\bf i}+1_{\overline{I}})-x({\bf i})-E_{\overline{I}}$.
\begin{lemma}\label{lem:i+I}
There exists a sequence as in Lemma \ref{lem:con2} with $x^*_0=
x({\bf i}+1_{\overline{I}})-x({\bf i})-E_{\overline{I}}$.
\end{lemma}
\begin{proof} First we show the existence of some $l^*_2$, with all its coefficient very large,
which  can be connected by a computation sequence to $\infty$ with properties
(a)-(b)-(c) of \ref{lem:con2}. For this,
consider the full subgraph supported by $\calj^*$. Since it is negative definite, it supports an
effective cycle $Z^*$ such that $(Z^*,E_j)<0$ for any $j\in \calj^*$. Consider any sequence $\{x^*_n\}_{n=0}^t$,
$x^*_{n+1}=x^*_n+E_{j(n)}$,  such that $x^*_0=0$ and $x^*_t=Z^*$. Then, there exists
$\ell_0\geq 1$ sufficiently large such that   for any $\ell\geq \ell_0$ and $n$ one has
$$\chi_k(x({\bf i})+E_{\overline{I}}+\ell Z^*+x^*_{n+1})
\geq \chi_k(x({\bf i})+E_{\overline{I}}+\ell Z^*+x^*_{n}).$$
Hence the sequence $\{\ell Z+x_n\}_{\ell \geq \ell_0,\, 0\leq n\leq t}$ connects $l_2^*=\ell_0Z^*$ with $\infty$ with the required properties.

Next, we connect $x({\bf i}+1_{\overline{I}})-x({\bf i})-E_{\overline{I}}$ with this $l^*_2$
via a sequence which satisfies (a)-(b)-(c) of Lemma \ref{lem:con2}. Its existence follows from the following statement:

For any $l^*>0$ supported by $\calj^*$ there exists at least one index $j\in|l^*|$ such that
$$\chi_k(x({\bf i}+1_{\overline{I}})+l^*-E_j)\leq \chi_k(x({\bf i}+1_{\overline{I}})+l^*).$$
Indeed, assume the opposite. Then $(E_j,l^*)\geq E_j^2+2$ for any $j\in|l^*|$.
Hence $(E_j,l^*+k_{can})\geq 0$, or $\chi_{can}(l^*)\leq 0$, which contradict the rationality of the subgraph
 supported by $\calj^*$.
\end{proof}

Finally,  by Proposition \ref{prop:G}(I) (applied for $\overline{I}=\overline{J}$ and
${\bf s'}={\bf s}({\bf i}, \overline{J})$),
the newly determined `upper' bound $l^*_2=x({\bf i}+1_{\overline{I}})-x({\bf i})-E_{\overline{I}}$
can be pushed down further to its support $\cals\ii$.
Hence \ref{prop:G}(I), \ref{lem:i+I} and \ref{lem:con4} imply the following.

\begin{corollary}\label{cor:con}
There exists a deformation contraction of $\phi^{-1}\ii$ to $R_{\ii}(0, E_{\cals\ii})$ along which $\chi_k$ is
non--increasing.  Moreover, its restriction induces a deformation
 retract from  $(\phi|_{S_N})^{-1}\ii$ to $S_N\cap R_{\ii}(0, E_{\cals\ii}  )$. Restricting further,
 it gives a deformation
 retract from  $\phi^*_N\ii$ to  $\Phi^*_N\ii$, where $\Phi^*_N\ii$ is the product of the cube $\ii$ with
  $$\psi(\phi^*_N\ii)\cap \{l^*\,:\, 0\leq l^*-\psi(x({\bf i}))\leq  E_{\cals\ii}\}.$$
\end{corollary}
Note that this last space  $\Phi^*_N\ii$ is now rather `small': it is contained in the cube
$(x({\bf i}),\overline{I}\cup \cals\ii)$.
Nevertheless, the $N$--filtration of this cube can be rather complicated!

The statement of the above corollary
means that if  $\Phi^*_N\ii$ is empty if and only if  $\phi^*_N\ii$ is empty,
and when  they are not empty then they
have the same homotopy type. Therefore, via (\ref{eq:leray3}), we need to show that
$$\Phi^*_N\ii \ \ \mbox{ {\it is non--empty and contractible}.}$$
\subsection{The non--emptiness of  $\Phi^*_N\ii$.}\label{ss:nonempty}
Recall that we fixed an integer $N$ and a   cube
$({\bf i},\overline{I})$ which belongs to $\overline{S}_N$. By Definition \ref{bek:331} and
Proposition~\ref{prop:G}(I)(b) this reads as
\begin{equation}\label{eq:NN}
 \chi_k(x({\bf i}+1_{\overline{J}}))=
  \chi_k(x({\bf i})+E_{\overline{J}\cup \cals({\bf i},\overline{J})})
 \leq N \ \
 \mbox{for every} \ \  \overline{J}\subseteq\overline{I}.
\end{equation}
The non-emptiness follows from the following statement.
 \begin{proposition}\label{th:nonempty} For any fixed cube $({\bf i},\overline{I})\in \overline{S}_N$
 there exists  a cycle in $L$ of the form $x({\bf i})+E_{\widetilde{\cals}\ii}$ such that
 $\widetilde{\cals}\ii\subseteq \cals\ii$ and
 $(x({\bf i})+E_{\widetilde{\cals}\ii},\overline{I}) \subseteq \Phi_N^{*}\ii$; that is
 \begin{equation}\label{eq:NN2}
 \chi_k(x({\bf i})+E_{\overline{J}\cup \widetilde{\cals}\ii})
 \leq N \ \
 \mbox{for every} \ \  \overline{J}\subseteq\overline{I}.
\end{equation}
 \end{proposition}

\begin{proof}
The proof is long, it fills all this subsection \ref{ss:nonempty}.
 It is an induction over the cardinality of $\calj$, respectively of $\overline{I}$.
At start we reformulate it by keeping only the necessary combinatorial data, and we
also perform three  reductions to simplify the involved combinatorial complexity.
We will also write $\widetilde{\cals}:=\widetilde{\cals}\ii$ for the wished cycle.

\bekezdes\label{bek:egy}  {\bf Starting the reformulation.} Define (cf. Proposition \ref{prop:G}(I))
\begin{equation}\label{eq:NN3}
N(G):= \max_{\overline{J}\subseteq \overline{I}}\chi_k(x({\bf i}+1_{\overline{J}}))=
 \max_{\overline{J}\subseteq \overline{I}} \chi_k(x({\bf i})+E_{\overline{J}\cup \cals({\bf i},\overline{J})}).
\end{equation}
$N(G)$ is the smallest integer $N$ for which (\ref{eq:NN}) is valid; hence it is enough to prove
Theorem \ref{th:nonempty}  only for $N=N(G)$. Note that $N(G)$ depends on $\ii$, though in its notation this is
not emphasized.

In fact, even the weight $\chi_{k}(x({\bf i}))$ ---
and partly the cycle $x({\bf i})$, cf. \ref{bek:sigma}, ---
are irrelevant in the sense that it is enough to treat a relative version of the statement.
Indeed, we can consider only
the value $\Delta N(G):=N(G)-\chi_{k}(x({\bf i}))$, which equals (use the last term of (\ref{eq:NN3})):
\begin{equation}\label{eq:red1}
\Delta N(G)=\max_{\overline{J}\subseteq \overline{I}}\, \Big(
 \chi_{can}(E_{\overline{J}\cup \cals({\bf i},\overline{J})})-
(E_{\overline{J}\cup \cals({\bf i},\overline{J})}\,,\, x({\bf i})+l'_{[k]})\Big).\end{equation}
 Then, cf. (\ref{eq:NN2}),  we have to find
$\widetilde{\cals}\subseteq \cals\ii$, such that  for any $\overline{J}\subseteq \overline{I}$ one has
  \begin{equation}\label{eq:NN4}
  \chi_{can}(E_{\overline{J}\cup \widetilde{\cals}})-
(E_{\overline{J}\cup \widetilde{\cals}}\,,\, x({\bf i})+l'_{[k]})\leq \Delta N(G).
\end{equation}
 Note also that for a reduced cycle $Z$ of $G$ (as $E_{\overline{J}\cup \cals({\bf i},\overline{J})}$
 or $E_{\overline{J}\cup \widetilde{\cals}}$),  $\chi_{can}(Z)$ is the number of components of $Z$,
 which sometimes will also be denoted by $\#(Z)$.

\vspace{2mm}

It is convenient to set the following notation.
For any vertex $j$ and $J\subseteq \calj$ set
$$\sigma_j:=1-(E_j,x({\bf i})+l'_{[k]}) \ \ \mbox{and} \ \
\sigma_j(J):=\sigma_j -(E_j,E_J).$$
By definition of $x({\bf i})$,  one has  $\sigma_j>0$ for any $j\in\scalj$.
Note also  that the information needed in (\ref{eq:red1}) and (\ref{eq:NN4}) about $x({\bf i})+l'_{[k]}$ can
be totally codified  by the integers $\sigma_j$. This permits to reformulate the statement of the Paragraph
\ref{bek:egy} into the following version:

\bekezdes\label{bek:sigma}{\bf Final Reformulation.}
Let $G$ be a connected graph
(e.g. a plumbing graph whose Euler decorations are deleted),
with $\calj=\ocalj\sqcup\scalj$, such that any two vertices of $\ocalj$ are not adjacent, and
with additional decorations
$\{\sigma_j\}_{j\in\calj}$ where $\sigma_j>0$ for $j\in\scalj$.
Fix $\overline{I}\subseteq \ocalj$.  For each $\overline{J}\subseteq \overline{I}$ we
define $\cals(\overline{J})$ as the minimal support in $\scalj$ such that for any $j\in \scalj\setminus \cals(\overline{J})$ one has $\sigma_j(\overline{J}
\cup\cals(\overline{J}))>0$. [Clearly, $\cals(\overline{J})$ corresponds to $\cals({\bf i}, \overline{J})$
in the original version, see also  \ref{prop:G}.]

The `modified' Laufer algorithm to find $\cals(\overline{J})$ (transcribed in the language
of $\sigma_j$'s) is the following. We construct the sequence of supports
$\{s_n\}_{n=0}^t$ by the next principle: $s_0=\emptyset$, and  if $s_n$ is already constructed and there exists some
$j(n)\in \scalj\setminus s_n$ such that
\begin{equation}\label{eq:degger}
\sigma_{j(n)}(\overline{J}\cup s_n)
=\sigma_{j(n)}-(E_{j(n)},E_{\overline{J}\cup s_n})
\leq 0\end{equation}
then take $s_{n+1}:=s_n\cup j(n)$; otherwise stop,
and set $t=n$. [This again follows from the fact that $(E_j, x({\bf i})+E_{\overline{J}\cup s_n}+l'_{[k]})>0$
if and only if  $\sigma_j(\overline{J}\cup s_n)\leq 0$.]
Note that $\cals(\emptyset)=\emptyset$.

Then  the statements  form  \ref{bek:egy} (hence what we need to show) read as follows.


For any $\overline{J}\subseteq \overline{I}$ set
\begin{equation}\label{eq:red2}
\Delta (\overline{J};G) :=
 \#(E_{\overline{J}\cup \cals(\overline{J})})+ \sum_{j\in
\overline{J}\cup \cals(\overline{J})}( \sigma_j-1), \ \  \mbox{and} \ \ \
\Delta N(G)=\max_{\overline{J}\subseteq \overline{I}}\,
\Delta (\overline{J};G).\end{equation}
Then there exists  $\widetilde{\cals}\subseteq \cals(\overline{I})$
 which, for any $\overline{J}\subseteq \overline{I}$, satisfies
\begin{equation}\label{eq:red3}
 \#(E_{\overline{J}\cup \widetilde{\cals}})+ \sum_{j\in
\overline{J}\cup \widetilde{\cals}}( \sigma_j-1)\leq \Delta N(G).\end{equation}

\vspace{2mm}

Before we formulate the reductions, we list some additional {\it properties of this setup}.

\bekezdes\label{bek:s_n}
\ {\bf (P1)} \
We analyze how the numerical invariants are modified along the computation sequence
$\{s_n\}_{n=0}^t$ of \ref{bek:sigma}.
Note that if (\ref{eq:degger}) occurs, since $\sigma_{j(n)}>0$, $j(n)$ should be adjacent to $\overline{J}\cup s_n$.
If it is adjacent to only one vertex of $\overline{J}\cup s_n$, then necessarily $\sigma_{j(n)}=1$.
Furthermore,  in any situation,  $\#(E_{\overline{J}\cup s_n})$ is decreasing by
$(E_{j(n)}, E_{\overline{J}\cup s_n})-1$. Therefore, the sequence
$a_n(\overline{J}):=\#(E_{\overline{J}\cup s_n})+
\sum_{j\in \overline{J}\cup s_n}(\sigma_j-1)$ is modified during this step by
$$a_{n+1}(\overline{J})-a_{n}(\overline{J})=
\sigma_{j(n)}-(E_{j(n)}, E_{\overline{J}\cup s_n})\leq 0.$$

\vspace{2mm}

\noindent {\bf (P2)} \
For any $\overline{J}\subseteq \overline{I}$ and vertex  $j\in \overline{I}\setminus \overline{J}$ one has
\begin{equation*}\label{eq:j_0}
\Delta(\overline{J}\cup j;G)=\Delta(\overline{J};G)+
\sigma_{j}- (E_j,E_{\cals(\overline{J})}).
\end{equation*}
The proof runs as follows. Let $\{s_n\}_{n=0}^t$ be the computation sequence for
$\cals(\overline{J})$. It can be considered as the first part of a sequence for $\cals(\overline{J}\cup j)$
too; let $\{s_n\}_{n=t+1}^{t'}$ be its continuation for $\cals(\overline{J}\cup j)$.
The coefficients $a_n(\overline{J})$ and $a_n(\overline{J}\cup j)$ for $n\leq t$ can be compared. Indeed,
$a_0(\overline{J}\cup j)=a_0(\overline{J})+\sigma_j$, and, similarly as in (P1),
$a_{t}(\overline{J}\cup j)=a_t(\overline{J})+\sigma_j-(E_j,E_{\cals(\overline{J})})$,
which is
the right hand side of the above identity (since $a_t(\overline{J})=\Delta (\overline{J};G)$).

   Next, we show that $a_n(\overline{J}\cup j)$ is constant for any further value $n\geq t$.
First take  $n=t$. Then $\sigma_{j(t)}-(E_{j(t)}, E_{\overline{J}\cup \cals(\overline{J})})>0$
(since $\cals(\overline{J})$ is completed), but
$\sigma_{j(t)}-(E_{j(t)}, E_{\overline{J}\cup \cals(\overline{J})\cup j})\leq 0$
(since $\cals(\overline{J}\cup j)$ is not completed). Hence $(E_j,E_{j(t)})=1$ and (using (P1) too)
$a_{t+1}(\overline{J}\cup j)-a_t(\overline{J}\cup j)=
\sigma_{j(t)}-(E_{j(t)}, E_{\overline{J}\cup \cals(\overline{J})\cup j})= 0$.

In general, set  $s^j_n:=s_n\setminus \cals(\overline{J})$, e.g. $s_t^j=\emptyset$.
 At every step,
by induction, $E_{j\cup s^j_n}$is connected, hence $(E_{j(n)},  E_{j\cup s^j_n})$ can
be at most one (since the graph contains no loops). Hence,
$\sigma_{j(n)}-(E_{j(n)}, E_{\overline{J}\cup \cals(\overline{J})})>0$,
and $\sigma_{j(n)}-(E_{j(n)}, E_{\overline{J}\cup \cals(\overline{J})\cup s^j_n\cup j})\leq 0$ imply
 $(E_{j(n)},  E_{j\cup s^j_n})=1$ and $a_{n+1}(\overline{J}\cup j)=a_n(\overline{J}\cup j)$.

\vspace{2mm}

\noindent {\bf (P3)}
Fix a vertex $\overline{j}\in\overline{I}$ with $\sigma_{\overline{j}}\geq 1$,  and
assume that  {\it for all  realizations} of  $\Delta N(G) $  as $\Delta(\overline{J},G)$
(as in (\ref{eq:red2})) one has  $\overline{J}\ni \overline{j}$.
Let $G_{-1}$ be the graph obtained from $G$ by replacing the decoration $\sigma_{\overline{j}}$ by
$\sigma_{\overline{j}}-1$. We claim that
\begin{equation}\label{eq:-1}
\Delta N(G_{-1})=\Delta N(G)-1.
\end{equation}
Indeed,  since $\{\sigma_j\}_{j\in\calj^*}$ is unmodified,
the support $\cals(\overline{J})$
for any $\overline{J}$  is the same determined in $G_{-1}$ or in $G$.
If $\overline{J}\not \ni \overline{j}$ then $\Delta(\overline{J},G_{-1})=\Delta (\overline{J},G)$ by (\ref{eq:red2}),
hence $\Delta (\overline{J},G_{-1}) <\Delta N(G)$.
If $\overline{J}\ni \overline{j}$ then
$\Delta (\overline{J},G_{-1})=\Delta (\overline{J},G)-1$ by the same
(\ref{eq:red2}). Since one such $\overline{J}$ realizes $\Delta N(G)$, the claim follows.

 \bekezdes\label{bek:ibar} {\bf First Reduction: $\overline{I}=\ocalj$.} \
Consider $\ocalj\setminus \overline{I}=\overline{I}^c$ and the graph $G\setminus \overline{I}^c$ obtained from the
original graph $G$ by deleting the vertices $\overline{I}^c$ and their adjacent edges.
The connected components of  $G\setminus \overline{I}^c$
do not interact  from the point of view of the statement of Proposition~
\ref{th:nonempty}. Indeed, the Laufer algorithm does not
propagate along the bad vertices  $\overline{I}^c$, and it is also enough to find supports  $\widetilde{\cals}$ for each component independently. Hence, {\it we may assume that $\overline{I}=\ocalj$}.

\bekezdes
{\bf Second Reduction: $\sigma_j>0$ for any $j$.} \
Consider the situation from \ref{bek:sigma} with $\overline{I}=\ocalj$, cf. \ref{bek:ibar}.
Assume that $\sigma_j\leq 0$ for some $j\in \overline{I}=\ocalj$, and consider the graph $G\setminus j$ obtained from
$G$ by deleting the vertex $j$ and its adjacent edges. Note the following facts:

$\bullet$ \ The maximum $\Delta N(G)$ in (\ref{eq:red2})
can be realized by a subset $\overline{J}$ which does not contain $j$.
In fact, for any $\overline{J}$ with $j\not\in\overline{J}$ one has
$\Delta(\overline{J}\cup j;G)\leq \Delta(\overline{J};G)$. Indeed, using the notations from \ref{bek:s_n},
$a_0(\overline{J}\cup j)\leq a_0(\overline{J})$; the sequence $s_n$ associated with
 $\overline {J}$ is good as the beginning of the sequence of $\overline{J}\cup j$, and during this inductive steps
 $a_n(\overline{J}\cup j)$ drops more than $a_n(\overline{J})$; and finally, if the sequence of
 $\overline{J}\cup j$ is longer, then its $a_{n}$--values decrease
 even more (cf. \ref{bek:s_n}).

$\bullet$ \  All the supports of type
 $\cals(\overline{J})$ definitely are included  in $G\setminus j$ (since are subsets of
${\mathcal J}^*$).

$\bullet$ \ If we find for each component of $G\setminus j$ some $\widetilde{\cals}$ satisfying the statements of the theorem for that component, then their union solves the problem for $G$ as well.

Therefore, having $G$ with some $\sigma_j\leq 0$, we can delete $j$ and continue to search for $\widetilde{\cals}$
for $G\setminus j$: that support will work for $G$ as well.

If we delete all vertices with
$\sigma_j\leq 0$ ($j\in \overline{I}$) then
 we arrive to a situation when $\sigma_j>0$ for any $j\in\overline{I}$, hence,
a posteriori, $\sigma_j>0$ for any $j\in\calj$.

\vspace{2mm}

{\it Note that the wished reformulated statement from \ref{bek:sigma}, even for all $\sigma_j=1$, when the problem depends purely on the shape of the graph, is far to be trivial.}

\bekezdes\label{bek:G^-} {\bf Third Reduction: $G=G^-$}. \ Assume $\overline{I}=\ocalj$,
 cf. \ref{bek:ibar}.
 Let $G^-$ be the minimal connected subgraph of $G$ generated  by the vertices $\overline{I}$.
Here the vertices $\calj(G^-)$ have an induced  disjoint decomposition into $\ocalj(G^-)=\ocalj$ and
$\scalj(G^-)=\calj(G^-)\cap \scalj$.
Moreover, each connected component of $G\setminus G^-$ is glued to $G^-$
via a unique $j\in \overline{I}$.

We claim that a solution $\widetilde{\cals}$ for $G^-$ provides a solution for $G$ too.
Indeed, for any $\overline{J}\subseteq \overline{I}$, the supports  $\cals_G(\overline{J})$ and
$\cals_{G^-}(\overline{J})$ generated in $G$, respectively in $G^-$ satisfy the following.

 $\overline{J}\cup\cals_G(\overline{J})$
can be obtained from $\overline{J}\cup \cals_{G^-}(\overline{J})$ by gluing some subtrees of
$G\setminus G^-$ along some elements of $\overline{J}$. These subtrees are maximal among
those connected subgraphs of $G$ (supported in $\calj^*\setminus \calj(G^-)$)
with all $\sigma_j=1$ and adjacent to $G^-$. In particular,
 $\overline{J}\cup \cals_{G^-}(\overline{J})\subseteq \overline{J}\cup\cals_G(\overline{J})$,
and their topological realizations are homotopy equivalent;
$\sigma_j=1$ for any $j\in  \cals_{G}(\overline{J})\setminus  \cals_{G^-}(\overline{J})$;
and the integers $\#E_{\overline{J}\cup \cals(\overline{J})}$
computed for $G$ and $G^-$ are the same.

Therefore,  $\Delta N(G)=\Delta N(G^-)$, and
a solution $\widetilde{\cals}$ for $G^-$ is a solution for $G$ too.

Hence, {\it we can assume that $G=G^-$}.\\

This ends the possible reductions/preparations  and we start the inductive argument.

\bekezdes {\bf The induction.}
The proof is based on inductive argument over $\sigma_j$-- decorated  graphs
(with  $\overline{I}=\ocalj$, $\sigma_j>0$ and $G=G^-$), where we will consider subgraphs (with
induced decorations $\sigma_j$), and eventually we will decrease the decorations $\{\sigma_j\}_{j\in\ocalj}$.

If $\overline{I}$ is empty then $\Delta N(G)=0$;  if
$\overline{I}$  contains exactly one element $j_0$,  then by (\ref{bek:G^-}) $G=\{j_0\}$ and
by (\ref{eq:red2}) $\Delta N(G)=\sigma_{j_0}$. In both cases
$\widetilde{\cals}=\emptyset$ answers the problem.
\bekezdes
The inductive step is based on the following picture.
Recall that $G$ agrees with the smallest connected subgraph generated by
$\ocalj$. Let $j_0\in\ocalj$ be one of its end--vertices (that is, a vertex which has only one adjacent vertex
in $G$). Denote that connected component of $G\setminus \ocalj$ which is adjacent to $j_0$ by $G^*_0$.

If $G\setminus \ocalj=G^*_0$ then all the vertices from $\ocalj$ are adjacent to $G^*_0$ and
$\ocalj=\overline{I}$ is {\it exactly}  the set of end--vertices of $G$. Then one verifies
 (use \ref{bek:s_n}(P2)) that

$\bullet$ \   $\Delta(\overline{J};G) $ is increasing function in $\overline{J}$, hence
  $\Delta N(G)=\Delta(\overline{I},G)$,  and

$\bullet$ \   $\#(E_{\overline{J}\cup \cals(\overline{I})})=
\#(E_{\overline{I}\cup \cals(\overline{I})})$, hence (\ref{eq:red3}) holds for $\widetilde{\cals}=
\cals(\overline{I})$.

\vspace{2mm}

Next, assume that $G\setminus \ocalj\not=G^*_0$. We may also assume (by a good choice of
$j_0$) that there is {\it only one} vertex $\overline{j}$ of $\ocalj$ which is simultaneously adjacent to
$G^*_0$ and to some other component of  $G\setminus \ocalj$. Let $\{j_0,j_1,\ldots, j_k,\overline{j}\}$
be the elements of $\ocalj$ which are adjacent to $G^*_0$. Then $j_0,j_1,\ldots, j_k$ are end--vertices of
$G$.
%
Let $G'$ be  obtained from $G$ be deleting $G^*_0$, $\{j_0,j_1,\ldots, j_k\}$
and all their adjacent edges.

Here is the schematic picture of $G$, where the vertices from $\scalj$ are not emphasized:

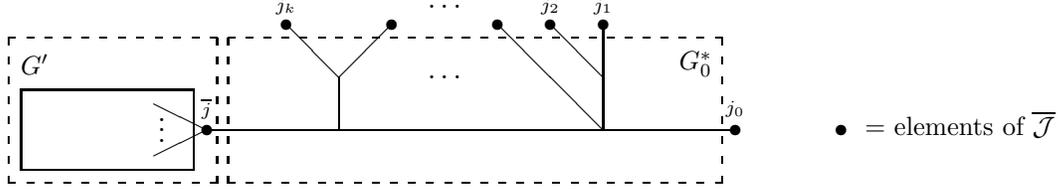
\begin{picture}(390,100)(10,-10)
\put(30,10){\framebox(65,30)}
\put(80,35){\line(2,-1){20}}
\put(80,15){\line(2,1){20}}
\put(83,28){\makebox(0,0){$\vdots$}}\put(190,45){\makebox(0,0){$\ldots$}}
\put(100,25){\circle*{4}}
\put(100,25){\line(1,0){200}}
 \put(300,25){\circle*{4}}
\put(150,25){\line(0,1){20}}
\put(150,45){\line(1,1){20}}
\put(150,45){\line(-1,1){20}}
\put(170,65){\circle*{4}} \put(130,65){\circle*{4}}
\put(250,25){\line(0,1){40}} \put(250,25){\line(-1,1){40}}
 \put(250,45){\line(-1,1){20}}
 \put(250,65){\circle*{4}} \put(230,65){\circle*{4}} \put(210,65){\circle*{4}}
\put(100,33){\makebox(0,0){\tiny{$\overline{j}$}}}
\put(300,33){\makebox(0,0){\tiny{$j_0$}}}
\put(130,72){\makebox(0,0){\tiny{$j_k$}}}
\put(250,72){\makebox(0,0){\tiny{$j_1$}}}
\put(230,72){\makebox(0,0){\tiny{$j_2$}}}
\put(190,72){\makebox(0,0){$\ldots$}}
\put(285,50){\makebox(0,0){\small{$G^*_0$}}}
\put(35,50){\makebox(0,0){\small{${G}'$}}}
\put(340,25){\circle*{4}}
\put(385,27){\makebox(0,0){\small{\mbox{= elements of $\ocalj$}}}}
\dashline[3]{3}(108,5)(295,5)\dashline[3]{3}(108,5)(108,60)
\dashline[3]{3}(108,60)(295,60)\dashline[3]{3}(295,5)(295,60)
\dashline[3]{3}(25,5)(104,5)\dashline[3]{3}(104,5)(104,60)
\dashline[3]{3}(25,60)(104,60)\dashline[3]{3}(25,5)(25,60)
\end{picture}

\vspace{5mm}

%

The inductive step splits in several cases ({\bf A} and {\bf B}, {\bf A} splits
into {\bf I} and {\bf II},
while {\bf I} has two subcases {\bf I.a} and {\bf I.b}).

\bekezdes\label{bek:j_02} {\bf A.} {\it \underline{ Assume that $\Delta N(G)$ in (\ref{eq:red2}) can be
realized by some $\overline{J}$
with  $j_0\not\in \overline{J}$.} }

\vspace{1mm}

Fix such a $\overline{J}$.
Since $\sigma_{j_0}\geq 1$ and 
$\Delta(\overline{J},G)\geq\Delta(\overline{J}\cup j_0,G)$, from \ref{bek:s_n}(P2) one gets
\begin{equation}\label{eq:delta2}
\mbox{$\sigma_{j_0}=1$ and  $j_0$ is adjacent to a vertex  of $\cals(\overline{J})$}.
\end{equation}
Assume that some $j_\ell$ ($1\leq \ell\leq k$) is not in $\overline{J}$. Then again by
$\Delta(\overline{J}\cup j_\ell,G)\leq \Delta(\overline{J},G)$ and \ref{bek:s_n}(P2) we get
that $\sigma_{j_\ell}=1$ and $j_\ell$ is adjacent to $\cals(\overline{J})$.
In particular, $\Delta(\overline{J}\cup j_\ell,G)=\Delta(\overline{J},G)$, and we can replace
$\overline{J}$ by $\overline{J}\cup j_\ell$. Hence, for uniform treatment, in such a situation
 we can always assume that
\begin{equation}\label{eq:jell}
\{j_1,\ldots, j_k\}\subseteq \overline{J}.
\end{equation}
Let $\cals^*_0$ be the support generated by $\{j_1,\ldots,j_k\}$ via the (reformulated)
Laufer's algorithm   \ref{bek:sigma};
then  $\cals^*_0\subseteq G^*_0$.

We will need another fact  too. Let $\overline{J}'$ be a subset of $\ocalj(G')$. Then
\begin{equation}\label{eq:''}
\Delta(\overline{J}',G')=\Delta(\overline{J}',G),
\end{equation}
that is, the $\Delta$--invariants of $\overline{J}'$
 in $G'$ and in $G$ are the same. Indeed, if $\overline{j}\not\in \overline{J}'$,
then the identity  is clear since $\overline{J}'$  generates the same supports
 $\cals(\overline{J}',G')=\cals(\overline{J}',G)$ in $G'$ and $G$. Otherwise,
  $\cals(\overline{J}',G)$ is the union of $\cals(\overline{J}',G')$
with the  maximal element of those
connected subgraph of $G^*_0$ which are  adjacent to $\overline{j}$ and  $\sigma_j=1$ for all
 their vertices $j$.

Now, our discussion bifurcates into two cases: {\it
whether $\overline{j}$ \, is adjacent to $\cals^*_0$ or not.}

\vspace{2mm}

\noindent
{\it {\bf I.} \underline{The case when  $\overline{j}$ is not adjacent to $\cals^*_0$.} }

\vspace{2mm}

We start with the following general statement, valid for any  $\overline{J}\subseteq \overline{I}$,
which does not contain $j_0$ but it contains $\{j_1,\ldots, j_k\}$. For such $\overline{J}$,
whenever $\overline{j}$ is not adjacent to $\cals^*_0$  one has:
\begin{equation}\label{eq:Delta}
\Delta(\overline{J},G)=\Delta(\{j_1,\ldots, j_k\},G)+\Delta(\overline{J}\cap G',G),
\end{equation}
where $\overline{J}\cap G'$ stands for $\overline{J}\cap \calj(G')$.
For its proof run first the Laufer algorithm for the vertices $ \{j_1,\ldots, j_k\}$ getting
$\cals^*_0$, then add the remaining vertices from $\overline{J}\cap G'$ and continue the
algorithm.

Therefore, for any $\overline{J}$ as in the assumption \ref{bek:j_02} (and with (\ref{eq:jell}))
we get that $\overline{J}\cap G'$ realizes $\Delta N(G')$. (Otherwise, we would be able to replace the subset
$\overline{J}\cap G'$ of $\overline{J}$ by another subset of $\ocalj\cap G'$ which would give larger
$\Delta(\overline{J}\cap G',G')=\Delta(\overline{J}\cap G',G)$, cf. also with (\ref{eq:''}), which would contradict
(\ref{eq:Delta}).)
Hence, (\ref{eq:Delta})  combined with (\ref{eq:''})  give:
\begin{equation*}\label{eq:DeltaN}
\Delta N(G)=\Delta N(G')+\Delta(\{j_1,\ldots, j_k\},G).
\end{equation*}

\noindent
{\bf I.a.} {\it \underline{Assume that $\Delta N(G') $ can be  realized by some $\overline{J}'$ in $G'$
which does not contain $\overline{j}$.}}

Then, we can apply the above statements for $\overline{J}=\overline{J}'\cup\{j_1,\ldots, j_k\}$.
Note that the Laufer algorithm runs in two independent regions cut by $\overline{j}$,
namely in $G^*_0$ and in  $G'\setminus \overline{j}$. Hence (\ref{eq:delta2}) guarantees that
$j_0$ is adjacent to $\cals^*_0$.

Furthermore, if $\widetilde{\cals}(G')$ is a support answering the problem
for $G'$, then $\widetilde{\cals}=\widetilde{\cals}(G') \cup \cals^*_0$ is a solution for $G$.
Note also that in this case $\cals^*_0$ coincides with the collection of components of
$\cals(\overline{I})$ sitting in $G^*_0$.

\vspace{2mm}

\noindent
{\bf I.b.} {\it \underline{Assume that all realizations of  $\Delta N(G') $  by some $\overline{J}'$ in $G'$
 contain $\overline{j}$.}}

Let $G'_{-1}$ be the graph obtained from $G'$ by replacing the decoration $\sigma_{\overline{j}}$ by
$\sigma_{\overline{j}}-1$. Then,  by \ref{bek:s_n}(P3),  we get
\begin{equation*}
\Delta N(G'_{-1})=\Delta N(G')-1.
\end{equation*}
By induction, one can find a support  $\widetilde{\cals}(G'_{-1})$ which solves the problem for  $G'_{-1}$.
Let $\calst$ be the connected (minimal) string in $G^*_0$ adjacent to both  $\overline{j}$ and $j_0$ (connecting them).

If $j_0$ is adjacent to $\cals^*_0$ then $\widetilde{\cals}= \widetilde{\cals}(G'_{-1})\cup \cals^*_0$
is a solution for $G$.

Otherwise  $\widetilde{\cals}= \widetilde{\cals}(G'_{-1})\cup \cals^*_0\cup \calst$
is a solution for $G$.

 \vspace{2mm}

\noindent
{\it {\bf II.} \underline{ The case when  $\overline{j}$ is adjacent to $\cals^*_0$. }}

\vspace{2mm}

Note  that in this case by the combinatorics of the choice of $j_0$ and by (\ref{eq:delta2})
we get that $j_0$ is adjacent to $\cals^*_0$ too.

We claim that for $G'_{-1}$ associated with the graph $G'$ and its vertex $\overline{j}$ one gets
\begin{equation*}\label{eq:DeltaN2}
\Delta N(G)=\Delta N(G'_{-1})+\Delta(\{j_1,\ldots, j_k\},G).
\end{equation*}
Moreover, $\widetilde{\cals}= \widetilde{\cals}(G'_{-1})\cup \cals^*_0$
is a solution for $G$.

\bekezdes\label{bek:j_03} {\bf B.} {\it \underline{Assume that for all  realizations of
$\Delta N(G)$ as $\Delta(\overline{J},G)$ one has
  $j_0\in \overline{J}$. }}

\vspace{2mm}

Replace in $G$ the decoration $\sigma_{j_0}$ by  $\sigma_{j_0}-1$, find a solution for
$G_{-1}$, then that solution works for $G$ too.
\end{proof}

This ends the proof of Proposition \ref{th:nonempty}. We continue with the contraction part.
\subsection{Additional properties of  $\widetilde{\cals}$.}\label{ss:tilde}
 Fix an integer $N$ and $\ii\in \overline{S}_N$ as in subsection \ref{ss:nonempty}.
 The cube $\ii$ determines the integer
  $N(G)=\max_{\overline{J}\subseteq \overline{I}}\chi_k(x({\bf i}+1_{\overline{J}}))$, cf.
 (\ref{eq:NN3}).
  Choose $\widetilde{J}\subseteq \overline{I}$ which realizes this maximum:
$N(G)=\chi_k(x({\bf i}+1_{\widetilde{J}}))$.   $N(G)$  is the smallest integer $N$
 for which $\ii\in \overline{S}_N$.

 Theorem \ref{th:nonempty} applied for $\ii$ and $N=N(G)$
 provides a cycle $x({\bf i})+E_{\widetilde{\cals}}$ with
 $\widetilde{\cals}\subseteq \cals\ii $  and
 \begin{equation}\label{eq:ssng}
 \chi_k(x({\bf i})+E_{\widetilde{\cals}}+E_{\overline{J}})\leq N(G) \ \ \mbox{for any $\overline{J}
 \subseteq \overline{I}$}.
 \end{equation}
In the next paragraphs we will list some additional properties of  $\widetilde{\cals}$ and $\widetilde{J}$.
\begin{lemma}\label{lem:tildeS} (a) \ $\chi_k(x({\bf i}) +E_{\widetilde{\cals}}+E_{\widetilde{J}})=N(G)$.
In particular, the weight of the cube $(x({\bf i})+E_{\widetilde{\cals}},\overline{I})$ is $N(G)$.

(b) (i) There exists a computation sequence $\{x_n\}_{n=0}^t$ with
$x_0=x({\bf i})+E_{\widetilde{J}}$ and $x_t=x({\bf i})+E_{\widetilde{J}}+E_{\widetilde{\cals}}$ such that
$\chi_k(x_{n+1})\leq \chi_k(x_n)$ for any $n$.

(ii) There exists a computation sequence $\{y_n\}_{n=0}^{t'}$ with
$y_0=x({\bf i})+E_{\widetilde{J}}+E_{\widetilde{\cals}}$ and $y_{t'}=
x({\bf i})+E_{\widetilde{J}}+E_{\cals\ii}$ such that
$\chi_k(x_{n+1})\geq \chi_k(x_n)$ for any $n$.

(c) Using the notation  $\sigma_j(J)$ from \ref{bek:sigma}, one has:
$$\begin{array}{lll} (i) & \sigma_j(\widetilde{\cals})\geq 0 & \ \mbox{if \ $j\in \widetilde{J}$}\\
(ii) & \sigma_j(\widetilde{\cals}
)\leq 0 & \ \mbox{if \ $j\not\in \widetilde{J}$}.\end{array}$$

\end{lemma}
\begin{proof} Note that
$$N(G)\stackrel{(1)}{=}
\chi_k(x({\bf i} +1_{\widetilde{J}}))\stackrel{(2)}{\leq}
\chi_k(x({\bf i}) +E_{\widetilde{\cals}}+E_{\widetilde{J}})
\stackrel{(3)}{\leq }N(G).$$
(1) follows from the definition of $N(G)$ and the choice of $\widetilde{J}$,
(2) from Lemma \ref{lem:laufb}, and (3) from Theorem \ref{th:nonempty} applied for $N=N(G)$. This proves  (a).
Identity (a) together with Proposition \ref{prop:G}(II) imply that $\widetilde{\cals}\subseteq
\cals({\bf i},\widetilde{J})$. Then there exists a  computation sequence connecting $x({\bf i})+E_{\widetilde{J}}$
with $x({\bf i})+E_{\widetilde{J}}+E_{\widetilde{\cals}}$  by  \ref{prop:G}(II), a sequence
connecting $x({\bf i})+E_{\widetilde{J}}+E_{\widetilde{\cals}}$  with
$x({\bf i})+E_{\widetilde{J}}+E_{\cals({\bf i},\widetilde{J})}$ by \ref{prop:G}(I), and finally, from
$x({\bf i})+E_{\widetilde{J}}+E_{\cals({\bf i},\widetilde{J})}$ to
$x({\bf i})+E_{\widetilde{J}}+E_{\cals({\bf i},\overline{I})}$ by \ref{prop:G}(III). This ends part (b).

Part (c) follows from (a) and equation (\ref{eq:ssng}) applied for $\widetilde{J}\setminus \{j\}$
(case $j\in \widetilde{J}$), respectively $\widetilde{J}\cup \{j\}$
(case $j\not\in \widetilde{J}$), and from the assumption \ref{ss:assumption}, which guarantees   $(E_j,E_{\widetilde{J}\setminus \{j\}})=0$.
\end{proof}

\bekezdes Let us recall what we already proved. For any fixed $\ii\in\overline{S}_N$ the space
 $\phi^{*}_N\ii$ is non-empty, cf. \ref{th:nonempty}, and it has the homotopy type of
 the product (cf. \ref{cor:con}):
 $$\Phi^{*}_N\ii=
 \psi(\phi^*_N\ii)\cap \{l^*\,:\, 0\leq l^*-\psi(x({\bf i}))\leq  E_{\cals\ii}\}
 \times \ii.$$
If $x\in \Phi^{*}_N\ii$ then  $x-x({\bf i})$ is reduced.  Moreover, $\Phi^{*}_N\ii$ has in it a distinguished $|I|$--dimensional
cube  $\{\psi(x({\bf i}))+E_{\widetilde{\cals}}\}\times \ii=(x({\bf i})+E_{\widetilde{\cals}},\overline{I})$.
Our goal is to construct a deformation retract from  $\Phi^{*}_N\ii$ to this cube (acting in the fiber direction).
This will be more complicated than the `standard' retractions
\ref{lem:con1}--\ref{lem:con2}--\ref{lem:con3}.
(Note that the point $x({\bf i})+E_{\widetilde{\cals}}+E_{\widetilde{J}}$
is not a $\chi_k$--minimal point of $\Phi^{*}_N\ii$,
it is maximal point in the direction $\ocalj$ and a minimal point in the direction $\calj^*$.)

To start with, we consider the connected components $\{G_\alpha\}_{\alpha\in A}$ of $\widetilde{\cals}$, and
the connected components $\{C_\beta\}_{\beta\in B}$ of $\cals\ii\setminus \widetilde{\cals}$.
During the contraction the supports $G_\alpha$ should be `added' and the supports $C_\beta$ should be `deleted'.
According to this, it is performed in several steps, during one step either we add one $G_\alpha$--type component, or we
delete one $C_\beta$--type component. At each step the fact that
which type is performed, or which $G_\alpha/C_\beta$ is manipulated is
decided by a technical `selection procedure'. This is the subject of the next Proposition, which will be applied at any
situation when the components $\{G_\alpha\}_{\alpha\in A'}$ still should be added and the
components $\{C_\beta\}_{\beta\in B'}$ still should be deleted: it chooses an element of $A'\cup B'$.
 The technical properties
associated with  the corresponding cases  will guarantee that the contraction stays below level $N$  of $\chi_k$.

Below, for any subset ${\calj'}\subseteq \ocalj$ and $i\in\calj^*$ we write $\calj'_i:=\{j\in \calj'\,:\, (E_i,E_j)=1\}$.
\begin{proposition}\label{prop:sel} {\bf (Selection Procedure)}
Fix subsets $A'\subseteq A$ and $B'\subseteq B$ such that $A'\cup B'\not=\emptyset$.
Then either there exists $\alpha\in A'$ such that
$$(i) \ \ \ \mbox{for every  $i\in |G_\alpha|$ and every $j\in \widetilde{J}_i$ one has
$\sigma_j(({\widetilde{\cals}}\setminus i )\cup \cup_{\beta\in B'}C_\beta ) >0$}\hspace{8mm}$$
or, there exists $\beta\in B'$ such that
$$(ii)  \ \ \mbox{for every  $i\in |C_\beta|$ and every
$j\in \overline{I}_i\setminus  \widetilde{J}$ one has
$\sigma_j(({\widetilde{\cals}}\cup i)\setminus \cup_{\alpha\in A'}G_\alpha)< 0$}.$$
\end{proposition}
\begin{proof}
Fix some $\alpha\in A'$ and assume that it does not satisfy (i). Then there exists $i_\alpha \in |G_\alpha|$
and $j_\alpha\in \widetilde{J}_{i_\alpha}$ such that
$\sigma_{j_\alpha}(({\widetilde{\cals}}\setminus i )\cup \cup_{\beta\in B'}C_\beta ) \leq 0$.
Note that $\sigma_{j_\alpha}({\widetilde{\cals}}\setminus i )=
\sigma_{j_\alpha}({\widetilde{\cals}} )+(E_{j_\alpha},E_{i_\alpha})>0$ by \ref{lem:tildeS}(c). These two combined
prove the existence of some $\beta\in B'$ and $i_\beta\in|C_\beta|$ with $(E_{j_\alpha}, E_{i_\beta})=1$.

Symmetrically, if for some $\beta\in B'$ (ii) is not true, then there exists $i_\beta\in |C_\beta|$ and
$j_\beta\in \overline{I}_{i_\beta}\setminus \widetilde{J}$ with
$\sigma_{j_\beta}(({\widetilde{\cals}}\cup i_\beta)\setminus \cup_{\alpha\in A'}G_\alpha)\geq  0$.
Since by \ref{lem:tildeS}(c) we have $\sigma_{j_\beta}({\widetilde{\cals}}\cup i_\beta)=
\sigma_{j_\beta}({\widetilde{\cals}})-(E_{j_\beta},E_{i_\beta})<0$, we get the existence of some
$\alpha\in A'$ and $i_\alpha\in |G_\alpha|$ with $(E_{j_\beta}, E_{i_\alpha})=1$.

Now the proof runs as follows. Start with any $\alpha\in A'$. If it satisfy (i) we are done. Otherwise,
as in the first paragraph,
we get  a $\beta$, such that $G_\alpha$ and $C_\beta$ are connected by a length two  path
having the middle vertex in $\widetilde{J}$.
If this $\beta $ satisfy (ii) we stop, otherwise we get by the second paragraph an $\alpha'$
such that $C_\beta$ and  $G_{\alpha'}$ are connected by a length two path whose middle vertex is not in
$\widetilde{J}$. Since the graph $G$  has no cycles,
$\alpha'\not=\alpha$. Then we continue the procedure with $\alpha'$. Either it satisfies (i) or
$G_{\alpha'}$ is connected with some $C_{\beta'}$ with $\beta'\not=\beta$. Continuing in this way, all the involved
 $\alpha$ indices, respectively  all the  $\beta$ indices are pairwise distinct because of the non--existence of
 a cycle in the graph. Since $A'\cup B'$ is finite, the procedure must stop.
\end{proof}

\subsection{Contraction of  $\Phi^{*}_N\ii$.}\label{ss:cont} We will drop the symbol $\ii$ from the notation
$\Phi^*_N\ii$: we  write simply $\Phi_N^*$. On the other hand, for any pair
$\emptyset\subseteq \cals_1\subseteq \cals_2\subseteq \cals\ii$, we define
$$\Phi^*_N(\cals_1,\cals_2):=
[\psi(\phi^*_N\ii)\cap
 \{l^*\,:\, E_{\cals_1}\leq l^*-\psi(x({\bf i}))\leq  E_{\cals_2}\}]
 \times \ii.$$
For example, $\Phi^*_N(\emptyset,\cals\ii)=\Phi^*_N$, while
$\Phi^*_N(\widetilde{\cals},\widetilde{\cals})=
\{(\psi(x({\bf i}))+E_{\widetilde{\cals}})\}\times \ii$, the cube on which we wish to contract $\Phi^*_N$.

If the Selection Procedure chooses some $\alpha'\in A'$ then we have to construct a deformation retract
$$c_{\alpha'}:\Phi^*_N( \bigcup_{\alpha\not\in A'}|G_\alpha|\,,\,
 \widetilde{\cals}\cup \bigcup_{\beta\in B'}|C_\beta|) \longrightarrow
 \Phi^*_N( \bigcup_{\alpha\not\in A'\setminus \alpha'}|G_\alpha|\,,\,
 \widetilde{\cals}\cup \bigcup_{\beta\in B'}|C_\beta|).$$
Otherwise, if some   $\beta'\in B'$ is chosen then we have to construct a deformation retract
$$c_{\beta'}:\Phi^*_N(\bigcup_{\alpha\not\in A'}|G_\alpha|\,,\,
 \widetilde{\cals}\cup \bigcup_{\beta\in B'}|C_\beta|) \longrightarrow
 \Phi^*_N( \bigcup_{\alpha\not\in A'}|G_\alpha|\,,\,
 \widetilde{\cals}\cup \bigcup_{\beta\in B'\setminus \beta'}|C_\beta|).$$
Their composition (in the selected order)  provides the wished deformation retract
$\Phi^*_N\to \Phi^*_N(\widetilde{\cals},\widetilde{\cals})$.
The two types of contractions have some asymmetries,
hence we will provide the details for both of them.

\bekezdes {\bf The construction of $c_{\alpha'}$. } Let $|G_{\alpha'}|=\{j_1,\ldots , j_t\}$.
By the properties of $\widetilde{J}$, cf. \ref{lem:tildeS}(b),
we have a computation sequence with $\chi_k$ non--increasing from
$x({\bf i})+E_{\widetilde{J}}$ to $x({\bf i})+E_{\widetilde{J}\cup\widetilde{\cals}}$.
Since the components $\{G_\alpha\}_\alpha$ do not interact, we can permute elements
belonging to different components $G_\alpha$, hence we may assume that the first part
completed the components $\cup_{\alpha\not\in A'}G_{\alpha}$, then we complete $G_{\alpha'}$
and the order $\{j_1,\ldots , j_t\}$ is imposed by the computation sequence. Therefore, for any $1\leq n\leq t$,
\begin{equation}\label{eq:CS}
\sigma_{j_n}(\widetilde{J}\cup \cup_{\alpha\not \in A'}|G_\alpha|\cup\{j_1,\ldots,j_{n-1}\})\leq 0.
\end{equation}

The contraction $c_{\alpha'}$ will be a composition $c_{\alpha',t}\circ \cdots \circ c_{\alpha',1}$, where
$c_{\alpha',n}$ corresponds to the completion of the cycles with $E_{j_n}$ ($1\leq n\leq t$):
\begin{align*}c_{\alpha',n}:\Phi^*_N( \bigcup_{\alpha\not \in A'}|G_\alpha|\cup\{j_1,\ldots,j_{n-1}\}\,,\,
& \widetilde{\cals}\cup \bigcup_{\beta\in B'}|C_\beta|) \longrightarrow \\
& \Phi^*_N(\bigcup_{\alpha\not \in A'}|G_\alpha|\cup\{j_1,\ldots,j_n\}
 \,,\,
 \widetilde{\cals}\cup \bigcup_{\beta\in B'}|C_\beta|)
 \end{align*}
defined as follows. Write $x=x({\bf i})+E_{\overline{J}}+l^*$ ($l^*$ is reduced) with
\begin{equation}\label{eq:SUPL}
\cup_{\alpha\not \in A'}|G_\alpha|\cup\{j_1,\ldots,j_{n-1}\}\subseteq
|l^*|\subseteq  \widetilde{\cals}\cup \bigcup_{\beta\in B'}|C_\beta|.\end{equation}
Then
$$c_{\alpha',n}(x)=\left\{\begin{array}{ll}
x & \ \mbox{if $ j_n\in |l^*|$},\\
x+E_{j_n} & \ \mbox{if $ j_n\not\in |l^*|$}.\end{array}
\right.$$
Note that for any $l^*$ as above with $|l^*|\not\ni j_n$, the inequality (\ref{eq:CS}) implies
\begin{equation}\label{eq:CS2}
\sigma_{j_n}(\widetilde{J}\cup |l^*|)\leq 0.
\end{equation}
Fix such an $l^*$ with $|l^*|\not \ni j_n$. Then,
for {\it  any} $\overline{J}\subseteq \overline {I}$, we have to prove
\begin{equation}\label{eq:NNN}
\chi_k(x({\bf i})+E_{\overline{J}}+l^*+E_{j_n})\leq N.
\end{equation}
Set $\overline{J}(l^*):=\{j\in\overline{I}\,:\, \sigma_j(|l^*|)>0\}$.
We claim that if (\ref{eq:NNN}) is valid for $\overline{J}(l^*)$ then it is valid for every $\overline{J}\subseteq
\overline{I}$. This follows from the next identity whose second term is
$\leq 0$ by the definition of $\overline{J}(l^*)$.
\begin{equation}\label{eq:MMM}\begin{split}
&\chi_k(x({\bf i})+E_{\overline{J}}+l^*+E_{j_n})-
\chi_k(x({\bf i})+E_{\overline{J}(l^*)}+l^*+E_{j_n})\\
&=\sum _{j\in \overline{J}\setminus \overline{J}(l^*)}\big[\sigma_j(|l^*|)-(E_j,E_{j_n})\big]
-
\sum _{j\in \overline{J}(l^*)\setminus \overline{J}}\big[\sigma_j(|l^*|)-(E_j,E_{j_n})\big].
\end{split}\end{equation}
On the other hand, using  Selection Procedure (and its notations) we get
$\widetilde{J}_{j_n}\subseteq \overline{J}(l^*)$. Indeed, by the choice of $\alpha'$ in \ref{prop:sel}(i),
for $j_n\in|G_{\alpha'}|$ and for any $j\in \widetilde{J}_{j_n}$ one has
$\sigma_j(\widetilde{\cals}\setminus j_n\cup\cup_{\beta\in B'}C_\beta)>0$.
Then
$\sigma_j(|l^*|)>0$ by the support condition (\ref{eq:SUPL}).
Then $\widetilde{J}_{j_n}\subseteq \overline{J}(l^*)$ implies:
\begin{equation}\label{eq:SIGIN}
\sigma_{j_n}(\overline{J}(l^*)\cup |l^*|)\stackrel{(1)}{\leq} \sigma_{j_n}(\widetilde{J}_{j_n}\cup |l^*|)
\stackrel{(2)}{=}
\sigma_{j_n}(\widetilde{J}\cup |l^*|)\stackrel{(3)}{\leq} 0.
\end{equation}
(1) follows from $\widetilde{J}_{j_n}\subseteq \overline{J}(l^*)$, (2) from
$(E_{j_n},E_{\widetilde{J}_{j_n}})=(E_{j_n},E_{\widetilde{J}})$, and (3)  from (\ref{eq:CS2}).
Therefore,
$$\chi_k(x({\bf i})+E_{\overline{J}(l^*)}+l^*+E_{j_n})-
\chi_k(x({\bf i})+E_{\overline{J}(l^*)}+l^*)=\sigma_{j_n}(\overline{J}(l^*)\cup |l^*|)\leq 0.$$
Since $\chi_k(x({\bf i})+E_{\overline{J}(l^*)}+l^*)\leq N$ (by induction),
(\ref{eq:NNN}) is valid for $\overline{J}(l^*)$.

\bekezdes {\bf The construction of $c_{\beta'}$ }. Let $|C_{\beta'}|=V_1 \cup V_2$, where $V_1:=|C_{\beta'}|\cap (\cals(\vasi,\widetilde{J})\setminus \widetilde\cals)$ and $V_2:=|C_{\beta'}|\cap (\cals\ii\setminus \cals(\vasi,\widetilde{J}))$. The Laufer computation sequence given by \ref{prop:G}(I) connecting $x(\vasi)+E_{\widetilde{J}}+E_{\widetilde\cals}$ with $x(\vasi)+E_{\widetilde{J}}+E_{\cals(\vasi,\widetilde{J})}$ gives an ordering on $V_1=\{j_1,\ldots,j_{t_s}\}$ with the property
\begin{equation}\label{eq:CS12}
\sigma_{j_n}(\widetilde{J}\cup \{j_1,\ldots,j_{n-1}\})=0
\end{equation}
for every $1\leq n \leq t_s$. Similarly, applying \ref{prop:G}(III) for $E_{\cals\ii \setminus \cals(\vasi,\widetilde{J})}$ we have an ordering on $V_2=\{j_{t_s+1},\ldots,j_t\}$ such that
\begin{equation}\label{eq:CS13}
\sigma_{j_n}(\widetilde{J}\cup \{j_1,\ldots,j_{n-1}\})\geq 0
\end{equation}
for every $t_s+1\leq n \leq t$.

The contraction $c_{\beta'}$ will be $c_{\beta',1}\circ \ldots \circ c_{\beta',t}$, where $c_{\beta',n}$ corresponds to the deletion of the cycles with $E_{j_n}$ ($1\leq n\leq t$), i.e.
\begin{align*}
c_{\beta',n}:\Phi^*_N( \bigcup_{\alpha\not \in A'}|G_\alpha|\,,\,
& \widetilde{\cals}\cup \bigcup_{\beta\in B'}|C_\beta|\setminus \{j_{n+1},\ldots,j_t\}) \longrightarrow
\\ & \Phi^*_N(\bigcup_{\alpha\not \in A'}|G_\alpha|
 \,,\,
 \widetilde{\cals}\cup \bigcup_{\beta\in B'}|C_\beta| \setminus \{j_{n},\ldots,j_t\})
 \end{align*}
defined in the following way. Write $x=x(\vasi)+E_{\widetilde{J}}+l^*$ with
\begin{equation}\label{eq:SUPL2}
\cup_{\alpha\not \in A'}|G_\alpha|\subseteq
|l^*|\subseteq  \widetilde{\cals}\cup \bigcup_{\beta\in B'}|C_\beta|\setminus \{j_{n+1},\ldots,j_t\},
\end{equation}
then
$$c_{\beta',n}(x)=\left\{\begin{array}{ll}
x & \ \mbox{if $ j_n\not\in |l^*|$},\\
x-E_{j_n} & \ \mbox{if $ j_n\in |l^*|$}.\end{array}
\right.$$
Fix such an $l^*$ with $j_n\in |l^*|$, then we have to prove
\begin{equation}\label{eq:NNN2}
\chi_k(x({\bf i})+E_{\overline{J}}+l^*-E_{j_n})\leq N
\end{equation}
for any $\overline{J}\subseteq \overline{I}$. In this case the inequalities (\ref{eq:CS12}) and (\ref{eq:CS13}) implies
\begin{equation}\label{eq:CS22}
\sigma_{j_n}(\widetilde{J}\cup |l^*|\setminus j_n)\geq 0.
\end{equation}

Here we set $\overline{J}(l^*):=\{j\in \overline{I}\,:\, \sigma_j(|l^*|)\geq 0 \}$. Then if (\ref{eq:NNN2}) is valid for $\overline{J}(l^*)$ then it is so for any $\overline{J}\subseteq \overline{I}$. Indeed,
\begin{equation}\label{eq:MMM2}\begin{split}
&\chi_k(x({\bf i})+E_{\overline{J}}+l^*-E_{j_n})-
\chi_k(x({\bf i})+E_{\overline{J}(l^*)}+l^*-E_{j_n})\\
&=\sum _{j\in \overline{J}\setminus \overline{J}(l^*)}\big[\sigma_j(|l^*|)+(E_j,E_{j_n})\big]
-
\sum _{j\in \overline{J}(l^*)\setminus \overline{J}}\big[\sigma_j(|l^*|)+(E_j,E_{j_n})\big]\leq 0,
\end{split}\end{equation}
by the definition of $\overline{J}(l^*)$.
By the selection of $\beta'$ via \ref{prop:sel}(ii), for $j_n\in |C_\beta'|$ and for any $j\in \overline{I}_{j_n}\setminus \widetilde{J}$ one has $\sigma_j(({\widetilde{\cals}}\cup j_n)\setminus \cup_{\alpha\in A'}G_\alpha)< 0$, hence $\sigma_j(|l^*|)<0$, in other words $\overline{J}(l^*)\subseteq \widetilde{J}$.
Finally, from (\ref{eq:CS22}) we can deduce the inequality
\begin{align*}
\chi_k(x({\bf i})+E_{\overline{J}(l^*)}+l^*-E_{j_n})- &
\chi_k(x({\bf i})+E_{\overline{J}(l^*)}+l^*)=\\
& -\sigma_{j_n}(\overline{J}(l^*)\cup |l^*|\setminus j_n)\leq -\sigma_{j_n}(\widetilde{J}\cup |l^*|\setminus j_n)\leq 0.\end{align*}

\section{Application. Series and the Seiberg--Witten invariants.}\label{s:5}

\subsection{Seiberg--Witten invariants.}
Recall that the Seiberg--Witten invariants of the oriented 3--manifold $M$
are rational numbers $\frsw_{\frs}(M)$ associated with the $spin^c$--structures $\frs$ of $M$.
In terms of Heegaard--Floer homology $$HF^+(M)=\oplus_{\frs\in Spin^c(M)}HF^+(M,\frs)=
\oplus_{\frs\in Spin^c(M)}\big(\calt_{d(M,\frs)}\oplus HF^+_{red}(M,\frs)\big)$$ (for details see
articles of Ozsv\'ath and Szab\'o, e.g. \cite{OSz,OSz7}) it can be recovered as
$$\frsw_{\frs}(M)={\rm rank}_{\Z} HF^+_{red,even}(M,\frs)-{\rm rank}_{\Z} HF^+_{red,odd}(M,\frs)
-d(M,\frs)/2.$$
In \cite{NSW}  is proved that the normalized Euler characteristics of the lattice cohomology
also agrees with the Seiberg--Witten invariant (note that the weight function
of \cite{NSW} is shifted compared with the present one; see the comment
after  \ref{NSW} as well). This reads as follows.
\begin{theorem}[\cite{NSW}]\label{th:ECH} For any characteristic element  $k$ one has
\begin{equation}\label{th:SW}
-\frsw_{-[k]}(M(G))-(k^2+|\calj|)/8= eu(\bH^*(G,k)),\end{equation}
where $eu(\bH^*(G,k)):=-\m_k+
\textstyle{\sum_q}(-1)^q\rank_\Z \bH^q_{red}(G,k)$.
\end{theorem}
To emphasize the role of the Reduction Theorem, let us also write
$$eu(\bH^*(\overline{L},\overline{w}[k])):=-\m_k+
\textstyle{\sum_q}(-1)^q\rank_\Z \bH^q_{red}(\overline{L},\overline{w}[k])$$
computed from the reduced lattice. Then, evidently,
 the right hand side of (\ref{th:SW}) can be replaced by $eu(\bH^*(\overline{L},\overline{w}[k]))$.
 Recall also, cf. \ref{corF1}, that $\bH^q_{red}(\overline{L},\overline{w}[k])=0$ for
 $q\geq \nu$.

\subsection{Multivariable series.}\label{ss:series}
The above statement connecting the Euler characteristic of the lattice cohomology
with the normalized Seiberg--Witten invariant can be lifted to the level of a
combinatorial `zeta function'.  This object,  associated with the  plumbing graph,
is defined as follows. It is the multivariable Taylor expansion
$\sum_{l'}p_{l'}{\bt}^{l'}$ at the origin of
\begin{equation}\label{eq:Zdef}
Z({\bf t})=\prod_{j\in \calj}(1-{\bf t}^{E_j^*})^{\delta_j-2}\end{equation}
where $\delta_j$ is the valency of the vertex $j$, and
for any $l'=\sum_j l_j E_j \in L'$ we write ${\bf t}^{l'}:=\prod_{j\in \calj} t_{j}^{l_j}$. This lives in
$\Z[[L']]$, the $\Z[L']$--submodule of formal power series $\Z[[\bt^{\pm 1/d}]]$ in variables $\{t_j^{\pm 1/d}\}_j$, where
$d=\det(-\frI)$.

$Z(\bt)$ has a  unique natural decomposition $Z({\bf t})=\sum_{h\in H}Z_{h}({\bf t})$, where
$Z_{h}({\bf t})=\sum_{[l']=h}p_{l'}{\bf t}^{l'}$. 
Sometimes we also write $Z_{l'}$ for $Z_{[l']}$.

Since the rational Lipman cone $\calS'$ is $\Z_{\geq 0}\langle E_j^*\rangle_j$,
definition ({\ref{eq:Zdef}) shows that
$Z(\bt)$ is supported in $\calS'$, hence $Z_{l'}(\bt)$ is supported in $(l'+L)\cap \calS'$.

$Z$ is also called the {\em combinatorial Poincar\'e series}. This is motivated by the following fact, for details
see \cite{CDGPs,CDGEq,NPS,NCL}.  We may consider the equivariant divisorial Hilbert series $\calH(\bt)$ of
a normal surface singularity $(X,0)$ with fixed resolution graph $G$.
The key point connecting $\calH(\bt)$ with the topology of the link $M$ and the
graph $G$ is introducing the series $\calP(\bt)=-\calH(\bt) \cdot \prod_j(1-t_j^{-1})\in \Z[[L']]$. Then $Z(\bt)$ is the
topological candidate for $\calP(\bt)$. They agree for several singularities, e.g. for splice quotients (see \cite{NCL}), which
contain all the rational, minimally elliptic or weighted homogeneous singularities.
Motivated by analytic properties  of $\calP$ (see e.g. \cite{NCL}),
the second author proved that $Z(\bt)$ also codifies the
Seiberg--Witten invariants of the link $M$, and it is  related in a subtle way with the
weight function of the lattice complex. We recall these facts next.

\begin{theorem}[\cite{NSW}]\labelpar{NSW} Fix  one of the elements $l'_{[k]}$.
Then the following facts hold.

(1)\begin{equation*}\label{eq:Z}
Z_{\lk}(\bt)=
\sum_{l\in L}\,\Big(\sum_{I\subseteq \mathcal{J}} (-1)^{|I|+1}w(l,I)\Big)\, \vast^{l+l'_{[k]}}.
\end{equation*}

(2) Fix some   $l\in L$ with  $l+\lk \in -k_{can}+{\rm interior}(\calS')$. Then
\begin{equation*}
\sum_{\overline{l}\in L,\, \overline l \not\geq l}p_{\overline{l}+\lk }=
\chi_{k_r}(l)+eu( \bH^*(G,k_r).\end{equation*}
Hence, the truncated summation from the left hand side admits
a multivariable Hilbert polynomial, where the non-free part is provided by the quadratic
weight function $\chi_{k_r}(l)$, while the free term is $eu( \bH^*(G,k_r)$, the expression
which appears in Theorem \ref{th:ECH} as well.
 \end{theorem}

\noindent [In \cite{NSW} $w(k)$ is defined as $-(k^2+|\calj|)/8$ for  $k \in Char$.
If $k=k_r+2l$ then $w(k)=\chi_{k_r}(l)-(k_r^2+|\calj|)/8$. The last constant can be neglected in the
 sum of (1) since $\sum_{I\subseteq \mathcal{J}}(-1)^{|I|}=0$.
The sum in (2) is finite since $Z$ is supported in $\Z_{\geq 0}\langle E_j^*\rangle_j$ and
 all the entries of $E^*_j$ are strict positive, cf. (\ref{eq:POS}).]

Our next goal is to show that the series introduced above, but now with reduced
variables, still preserves all these properties; namely,  it can be recovered from the
reduced weighted cubes and it contains  all the information about the Seiberg--Witten invariant.

\subsection{Definition. The reduced zeta function.}
Recall that $\calj=\overline{\calj}\sqcup \calj^*$, where $\overline{\calj}$ is an
index set containing all the bad vertices. Let $\phi:L\to \overline{L}$ be the projection to the
$\overline{\calj}$--coordinates.
We also write $\overline{\bt}=\{t_j\}_{j\in\overline{\calj}}$ for the monomial variables
associated with $\overline{L}$, and
$\overline{\bt}^{\vasi}=\prod_{j\in\overline{\calj}} t_{j}^{i_j}$ for $\vasi=(i_1,\dots,i_{\nu})\in \overline{L}$.
Therefore,  $\bt^{l'}|_{t_j=1, \,\forall \,j\in \calj^*}=\overline{\bt}^{\phi(l')}$.
For any $h\in H$ set
\begin{equation*}
\overline Z_{h}(\overline{\bt}):=Z_{h}(\vast)\vert_{t_j=1, \, \forall j\in \mathcal{J}^*}.\end{equation*}
[We warn the reader that the reduced `non--decomposed' series $Z(\vast)\vert_{t_j=1, \,
\forall  j\in \mathcal{J}^*}$
usually does not contain sufficient information to reobtain each term
$\overline Z_{h}(\overline{\bt})$ ($h\in H$) from it.]

Fix one $\lk$, and write  $\overline Z_{\lk}(\overline{\bt})=
\sum_{\vasi\in \overline{L}} \overline p_{\vasi+\phi(l'_{[k]})}\overline{\bt}^{\vasi+\phi(l'_{[k]})}$.
Moreover, let $\overline \calS_{k}'$ be the projection of $\calS'\cap (l'_{[k]}+L)$.
Then $\overline Z_{k}(\overline{\bt})$ is supported on
$\overline \calS_{k}'$, and
for any $\vasi$ the sum  $ \sum_{\vasi'\ngeq \vasi}\overline p_{\vasi'+\phi(l'_{[k]})}$ is  finite
(properties inherited from $Z$).  Note that $\overline{\calS}:=\phi(\calS'\cap L)$ is a semigroup, and
$\overline \calS_{k}'$ is an $\overline{\calS}$--module.

\begin{theorem}\labelpar{thm:redZ}
Let $(\overline{L},\overline{w}[k])$ be as in \ref{bek:331}.
Then

(1) \begin{equation*}
  \overline Z_{\lk}(\overline{\bt})=
  \sum_{\vasi \in \overline L}\Big(\sum_{\overline{I}
  \subseteq \overline{\mathcal J}}(-1)^{|\overline{I}|+1}
  \overline{w}(\vasi,\overline{I})\Big)\overline{\bt}^{\vasi+\phi(l'_{[k]})}.
 \end{equation*}

(2)
 There exists $\vasi_0 \in \ocalS $ (characterized in the next Lemma
\ref{ilem})  such that for any $\vasi \in \vasi_0+\ocalS$
 $$ \sum_{\vasi'\ngeq \vasi}\overline p_{\vasi'+\phi(\lk)}=\overline{w}(\vasi)+
 eu(\bH^*(\overline{L},\overline{w}[k])).$$
Here $\overline{w}(\vasi)$ is a quasi--polynomial (cf. \ref{rem:QP}),
and   $eu(\bH^*(\overline{L},\overline{w}[k]))$ equals\,
 $eu( \bH^*(G,k_r))$.
\end{theorem}

\begin{proof} (1)
We abbreviate $k_r$ by $k$ and $\overline{w}[k]$ by $\overline{w}$.
By \ref{eq:Z}(1) we get
$$\overline{Z}_{\lk}(\overline{\vast})=
\sum_{\vasi\in\overline{L}}\sum_{\overline{I}\subseteq \overline{\mathcal J}}(-1)^{|\overline{I}|+1}
\Big( \sum_{l^*\in L^*} \sum_{I^*\subseteq \mathcal{J}^*} (-1)^{|I^*|}w(x(\vasi)+l^*,\overline{I}\cup I^*)
 \Big)\,\overline{\vast}^{\vasi+\phi(l'_{[k]})},$$
 where $L^*\subset L$ is the sublattice of $\calj^*$--coordinates.
For a fixed $\vasi$ and $\overline{I}\subseteq \overline{\mathcal{J}}$, denote the coefficient in the last bracket by $T=T(\vasi,\overline{I})$.  Then we have to show that $T=\overline{w}(\vasi,\overline{I})$.

We define a weighted lattice $(L^*,w^*)$ as follows: the weight
 of a cube $(l^*,I^*)$ in $L^*$ is  $w^*(l^*,I^*):=w(x(\vasi)+l^*,\overline{I} \cup I^*)$
(hence it  depends on $(\vasi,\overline{I})$).
This is a compatible weight function on $L^*$ since $w$ is so, moreover
$T= \sum_{l^*\in L^*}\sum_{I^*\subseteq \mathcal{J}^*}(-1)^{|I^*|}w^*(l^*,I^*)$.

Note also that for any fixed ${\bf i}$ there are only finitely many
 $l^*\in L^*$ for which $x({\bf i})+l^*\in\calS'$
(use (\ref{eq:POS})).  Hence,  the sum in $T$ is finite. Therefore,
(cf. \ref{ss:2} and \ref{ss:3}),  we can find a `large' rectangle
$R^*=R^*(l_1^*,l_2^*)=\{l^*\in L^* \,:\, l_1^*\leq l^*\leq l_2^*\}$ with certain $l_1^*$ and $l_2^*$
such that
$$T=\sum_{l^*\in R^*}\sum_{I^*\subseteq \mathcal{J^*}}(-1)^{|I^*|}w^*(l^*,I^*) \ \ \ \mbox{and} \ \ \
\bH^*(L^*,w^*)=\bH^*(R^*,w^*).$$
Using the result and methods of \cite[Theorem 2.3.7]{NSW}, for the counting function
$\mathcal{M}(t):=\sum_{l^*\in R^*}\sum_{I^*\subseteq \mathcal{J^*}}(-1)^{|I^*|}t^{w^*(l^{*},I^*)}$ we have
$$\lim_{t\rightarrow 1}\frac{\mathcal{M}(t)-t^{\min( w^*\vert_{R^*})}}{1-t}=\sum_{q\geq 0}(-1)^q
\rank_\Z(\bH^q_{red}(R^*,w^*)).$$
The Reduction Theorem \ref{red} and its proof says that $(L^*,w^*)$ has vanishing reduced cohomology, in particularly
$\bH^q_{red}(R^*,w^*)=0$ for any $q\geq 0$. Hence
 $$T=\tiny{\frac{d\mathcal{M}(t)}{dt}}\big|_{t=1}=\min( w^*\vert_{R^*})=\min_{l^*\in L^*} \{w(x({\bf i})+l^*,\overline{I})\}=
 \min_{l^*\in L^*}\max_{\overline{J}\subseteq \overline{I}} \{\chi_k(x({\bf i})+l^*+E_{\overline{J}})\}.$$
By Lemma \ref{lem:laufb} $\chi_k(x({\bf i})+l^*+E_{\overline{J}})\geq \chi_k(x({\bf i}+1_{\overline{J}}))$, hence
\begin{equation}\label{eq:minmax}
\max_{\overline{J}\subseteq \overline{I}}
\chi_k(x({\bf i})+l^*+E_{\overline{J}})\geq
\max_{\overline{J}\subseteq \overline{I}}
\chi_k(x({\bf i}+1_{\overline{J}}))=\overline{w}({\bf i},\overline{I}).\end{equation}
But, by Lemma \ref{lem:tildeS}(a) (for notations see also \ref{ss:tilde}), the minimum
over $l^*$ of the left hand side is realized for $l^*=E_{\widetilde{\cals}}$ with equality
in (\ref{eq:minmax}),  hence
$T=\overline{w}({\bf i},\overline{I})$.

\vspace{2mm}

We start the proof of part (2) by the following lemma
 (the analogue for $(\overline{L},\overline{w})$
  of Lemmas \ref{lem:con2} and \ref{lem:i+I}),
which identifies $\vasi_0$.

\begin{lemma}\label{ilem}
(a) Fix $l+\lk\in \calS'$ and take the projection
$\vasi:=\phi(l)$. Then $x(\vasi)+\lk \in \calS'$, hence
$\overline{w}(\vasi+1_j)>\overline{w}(\vasi)$
for every $j\in \overline{\mathcal{J}}$.

(b) There exists $\widetilde{\vasi} \in  \overline \calS$ such that for any $\vasi\in \widetilde{\vasi}+
\overline \calS$ one has a
sequence $\{\vasi_n\}_{n\geq 0}\in\overline \calS$ with
\begin{itemize}
\item[(i)] $\vasi_0=\vasi$, $\vasi_{n+1}=\vasi_n+1_{j(n)}$ for certain
$j(n)\in \overline{\mathcal{J}}$, and all entries of $\vasi_n$ tend to infinity as
$n\rightarrow\infty$;
\vspace{1mm}
\item[(ii)] for any  $n$ and  $0\leq \vasi_n'\leq \vasi_n$ with the same  $j(n)$-th coefficients, one has
$$\overline{w}(\vasi_n'+1_{j(n)})>\overline{w}(\vasi_n').$$
\end{itemize}

\end{lemma}
\begin{proof} (a) Since $l+\lk$ satisfies conditions (a)-(b) of \ref{lemF1}
in the definition of $x(\vasi)$, by the minimality of $x(\vasi)$ we get
that $l-x(\vasi)$ is effective and is supported on $\calj^*$.
Hence, $(x(\vasi)+\lk,E_j)\leq (l+\lk,E_j)\leq 0$ for any $j\in \ocalj$.
The last inequality follows from  \ref{propF1}.

(b) The negative definiteness of the intersection from guarantees the existence of
$\widetilde{\vasi}$ with (i). For (ii) note that if $\vasi=\phi(l)$ as in (a), and $0\leq \vasi'\leq \vasi$,
such that their $j$-entries agree, then
automatically $\overline{w}(\vasi'+1_{j})>\overline{w}(\vasi')$. Indeed,
$x(\vasi)-x(\vasi') $ is effective and supported on $\calj\setminus j$, hence
 $(x(\vasi')+\lk,E_j)\leq (x(\vasi)+\lk,E_j)\leq 0$ and \ref{propF1} applies again.
\end{proof}

We fix an $\vasi$ as in Lemma \ref{ilem}(b). Then
similarly as in subsection \ref{ss:3}, one obtains
\begin{equation}
\bH^*(\overline{L},\overline{w})\cong \bH^*(R(0,\vasi),\overline{w}),
\end{equation}
where $R(0,\vasi)=\{\vasi'\in \overline{L} \, : \, 0\leq \vasi'\leq \vasi \}$.
In particularly, if we set
$$\cale(R(0,\vasi)):=\sum_{(\vasi',\overline{I})\subseteq R(0,\vasi)}(-1)^{|\overline{I}|+1}\overline{w}(\vasi',\overline{I})$$
(sum over all the cubes of $R(0,\vasi)$),
then \cite[Theorem 2.3.7]{NSW} ensures that
\begin{equation}\label{cale}
\cale(R(0,\vasi))=eu(\bH^*(R(0,\vasi),\overline{w})).
\end{equation}
In the sequel we follow closely the
proof of Theorem \ref{NSW}(2) from  \cite[Theorem 3.1.1]{NSW}).

We choose a computation sequence $\{\vasi_n\}_{n\geq 0}$ as in \ref{ilem}
and set $R':=\{ \vasi' \in \overline L \ : \vasi'\geq 0 \ \mbox{and} \ \exists \,
j\in \overline{\calj} \ \mbox{with} \
(\vasi'-\vasi)_j\leq 0\}$.  $R'$ is
not finite, but $R'\cap  \ocalS_{k}'$ is a finite set.  Fix $\widetilde n$ so that $R'\cap
\ocalS_{k}'\subseteq R(0,\vasi_{\widetilde n})$, and define
$R'(\widetilde n):=R'\cap R(0,\vasi_{\widetilde n})$, \
$\partial_1 R'(\widetilde n):=R'\cap R(\vasi,\vasi_{\widetilde n})$,
and
$$\partial_2 R'(\widetilde n):=\{\vasi'\in  R'(\widetilde n)\,:\,
\exists j\in\overline{\calj} \ \mbox{with } \ (\vasi'-\vasi_{\widetilde n})_j=0\}.$$
 Then by part (1) of the theorem we have
$$\sum_{\vasi'\ngeq \vasi}\overline p_{\vasi'+\phi(\lk)}=\cale(R'(\widetilde n))-\cale
(\partial_1 R'(\widetilde n)\cup \partial_2 R'(\widetilde n)).$$
The right hand side is simplified as follows.
First, notice that we may find $\widetilde n$ sufficiently high
in such a way, that if we choose a sequence
$\{{\bf j}_m\}_{m=0}^t$ from ${\bf j}_0=0$ to ${\bf j}_t=\vasi$
with ${\bf j}_{m+1}={\bf j}_m+1_{j(m)}$,
we have the following property:
\begin{center}
\noindent for every ${\bf j}' \in \partial_2 R'(\widetilde n)$ with ${\bf j}'\geq {\bf j}_m$ and
$({\bf j}')_{j(m)}=({\bf j}_m)_{j(m)}$ one has
$\overline w({\bf j}'+1_{j(m)})\leq \overline w({\bf j}')$.
\end{center}
Indeed, $(x({\bf j}')+\lk,E_{j(m)})$ is increasing in ${\bf j}'$ with fixed $j(m)$-th coefficient.
(Any  ${\bf j}' \in \partial_2 R'(\widetilde n)$
has `large' entries corresponding to coordinates
 $j$ when  $({\bf j}'-i_{\widetilde n})_j=0$, and `small' entry corresponding to $j(m)$.
 Hence, when we increase the $j(m)$-th  entry by one,
the positivity of the quantities $(x({\bf j}')+\lk,E_{j(m)})$ is guaranteed  by the presence of `large' entries.)

Therefore, using the sequence $\{{\bf j}_m\}$ and
\ref{lem:con3}, there exists a contraction of $\partial_2 R'(\widetilde n)$ to
$\partial_1 R'(\widetilde n)\cap \partial_2 R'(\widetilde n)$ along which $\overline w$ is
non--increasing.
Then similarly as in (\ref{cale}), one gets $\cale(\partial_2 R'(\widetilde n))=
\cale(\partial_1 R'(\widetilde n)\cap \partial_2 R'(\widetilde n))$,
hence $\cale (\partial_1 R'(\widetilde n)\cup \partial_2 R'(\widetilde n))=
\cale(\partial_1 R'(\widetilde n))$ too.

Next, we claim that $\cale(R'(\widetilde n))=\cale(R(0,\vasi))$. Indeed,
using induction on the sequence $\{\vasi_n\}_{0\leq n<\widetilde n}$,
it is enough to show that $\cale(R'(n))=\cale(R'(n+1))$. This follows from
 \ref{ilem}, since for all $\overline I$ containing $j(n)$ and each
$(\vasi',\overline I)\in R'(n+1)\setminus R'(n)$ we have
$$\overline \omega(\vasi',\overline I)=\overline \omega(\vasi'+1_{j(n)},\overline I\setminus j(n)).$$
This ensures a combinatorial cancelation in the sum $\cale(R'(n+1))$, or an isomorphism in the
 corresponding lattice cohomologies, which gives the expected
equality.

With the same procedure applying to $\partial_1 R'(\widetilde n))$ we deduce the equality
$\cale(\partial_1 R'(\widetilde n))=\cale(\partial_1 R'(0))=-\overline \omega(\vasi)$. Hence the
identity follows.
\end{proof}

\begin{remark}\label{rem:QP}
The fact that $\overline{w}(\vasi)$ is a quasi-polynomial can be seen as follows.
Choose  $l\in \calS'\cap L$, $l=(\overline{l},l^*)$, such that $(l,E_j)=0$ for any $j\in \calj^*$.
Then one checks that $x(\vasi+n\overline{l})=x(\vasi)+nl$ for any $n\in \Z_{\geq 0}$, hence
$\overline{w}(\vasi+n\overline{l})=\chi_{k_r}(x(\vasi)+nl)$ is a polynomial in $n$.
\end{remark}

\begin{example} \cite{OSZINV}\label{ex:Seifert}
In the reduced case, the expression $\overline{w}(\vasi)$ usually is a rather complicated
arithmetical quasi--polynomial. E.g., assume that $G$ is a star--shaped graph whose central vertex has
Euler decoration $b$ and the legs have Seifert invariants $(\omega_j,\alpha_j)_{j=1}^\ell$,
$0<\omega_j<\alpha_j$, ${\rm gcd}(\omega_j,\alpha_j)=1$, $\ell\geq 3$.
We fix the central vertex as the unique
bad vertex. Then, by Reduction Theorem, the lattice cohomology is completely determined by the
sequence $\{\overline{w}(i)\}_{i\geq 0}$.
Moreover, in the case of the canonical $spin^c$--structure, for any $i\geq 0$ one has
$$\overline{w}(i)=\sum_{0\leq k<i}N(k), \ \ \ \mbox{where} \ \
N(k):=1-kb-\sum_{j}\big\lceil \,k\omega_j/\alpha_j\,\big\rceil.$$
In this case $\overline{Z}_0(t)=\sum_{k\geq 0}\max\{0,  N(k)\}\,t^k$ and
$eu(\bH^*(G,can))=\sum_{k\geq 0}\max\{0, - N(k)\}$.
\end{example}

\section{Example}

Consider the following graph

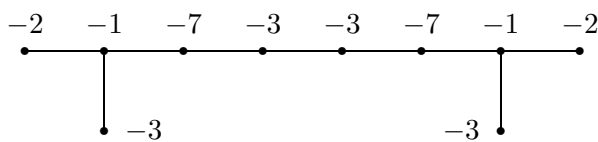
\begin{center}
\begin{picture}(140,70)(80,20)
\put(110,60){\circle*{3}}
\put(140,60){\circle*{3}}
\put(170,60){\circle*{3}}
\put(200,60){\circle*{3}}
\put(80,60){\circle*{3}}
\put(50,60){\circle*{3}}
\put(230,60){\circle*{3}}
\put(260,60){\circle*{3}}
\put(50,60){\line(1,0){210}}
\put(80,60){\line(0,-1){30}}
\put(80,30){\circle*{3}}
\put(230,60){\line(0,-1){30}}
\put(230,30){\circle*{3}}
\put(50,70){\makebox(0,0){$-2$}}
\put(80,70){\makebox(0,0){$-1$}}
\put(110,70){\makebox(0,0){$-7$}}
\put(140,70){\makebox(0,0){$-3$}}
\put(170,70){\makebox(0,0){$-3$}}
\put(200,70){\makebox(0,0){$-7$}}
\put(230,70){\makebox(0,0){$-1$}}
\put(260,70){\makebox(0,0){$-2$}}
\put(95,30){\makebox(0,0){$-3$}}
\put(215,30){\makebox(0,0){$-3$}}
\end{picture}
\end{center}
and its associated negative definite plumbed 3--manifold $M$. Notice that $M$ can be realized as the
link of the Newton
non--degenerate hypersurface singularity with equation $x^{13}+y^{13}+x^2y^2+z^3=0$.
In particular, $k_{can}$ is integral, and the weight function
$\chi_{can}(l)=-(l,l+k_{can})/2$ is fixed by the symmetry $l\mapsto -l-k_{can}$.

In the sequel we will calculate the lattice cohomology of $M$ 
associated with $k_{can}$. 
We can choose the two nodes for bad vertices. Then Reduction Theorem \ref{red} implies that
$\bH^*(M,k_{can})\cong \bH^*(\R_{\geq0}^2,\overline w)$, where $\overline w(i,j):=\chi_{can}(x(i,j))$ for any $(i,j)\in \Z_{\geq0}^2$. Using
Proposition \ref{propF1}, one can calculate the expressions
$$\overline w(i+1,j)-\overline w(i,j)=1+i-\lceil (53i+j)/351\rceil-\lceil i/2\rceil-\lceil i/3\rceil$$
$$\overline w(i,j+1)-\overline w(i,j)=1+j-\lceil (i+53j)/351\rceil-\lceil j/2\rceil-\lceil j/3\rceil.$$
Moreover, $R(0,(14,14))=\{(i,j)\in \R_{\geq0}^2 \ : \ (i,j)\leq (14,14)\}$ contains all the lattice cohomological
information. Indeed, one can prove that $L$ can be contracted to $R(0,-\K)$ along which $\chi_{can}$ is non--
increasing (see e.g. \ref{propF2} and \ref{lem:con2}). This also implies that after reduction enough to look at
$\phi(R(0,-\K))=R(0,(14,14))$, since $\phi(-\K)=(14,14)$.

We consider the picture below, illustrating the weighted lattice structure of $R(0,(14,14))$. (The lattice
point $(0,0)$ is at the lower left corner, $i$ increases in the horizontal direction, while $j$ increases in the
vertical one.)

\begin{center}
\begin{picture}(0,160)(120,-5)
\put(15,0){\small$0$}\put(30,0){\small$1$}\put(45,0){\small$0$}\put(60,0){\small$0$}\put(75,0){\small$0$}
\put(90,0){\small$0$}\put(105,0){\small$0$}\put(120,0){\small$1$}\put(135,0){\small$0$}\put(150,0){\small$0$}
\put(165,0){\small$0$}\put(180,0){\small$0$}\put(195,0){\small$0$}\put(210,0){\small$1$}\put(225,0){\small$1$}
\put(12,-2){\line(1,0){11}}\put(23,-2){\line(0,1){10}}\put(23,8){\line(-1,0){11}}\put(12,8){\line(0,-1){10}}

\put(15,10){\small$1$}\put(30,10){\small$1$}\put(45,10){\small$0$}\put(60,10){\small$0$}\put(75,10){\small$0$}
\put(90,10){\small$0$}\put(105,10){\small$0$}\put(120,10){\small$1$}\put(135,10){\small$0$}\put(150,10){\small$0$}
\put(165,10){\small$0$}\put(180,10){\small$0$}\put(195,10){\small$0$}\put(210,10){\small$1$}\put(225,10){\small$1$}

\put(15,20){\small$0$}\put(30,20){\small$0$}\put(38,20){\small$-1$}\put(53,20){\small$-1$}\put(68,20){\small$-1$}
\put(83,20){\small$-1$}\put(98,20){\small$-1$}\put(120,20){\small$0$}\put(128,20){\small$-1$}\put(143,20){\small$-1$}
\put(158,20){\small$-1$}\put(173,20){\small$-1$}\put(188,20){\small$-1$}\put(210,20){\small$0$}\put(225,20){\small$0$}
\put(36,18){\line(1,0){76}}\put(112,18){\line(0,1){50}}\put(112,68){\line(-1,0){76}}\put(36,68){\line(0,-1){50}}
\put(126,18){\line(1,0){76}}\put(202,18){\line(0,1){50}}\put(202,68){\line(-1,0){76}}\put(126,68){\line(0,-1){50}}

\put(15,30){\small$0$}\put(30,30){\small$0$}\put(38,30){\small$-1$}\put(53,30){\small$-1$}\put(68,30){\small$-1$}
\put(83,30){\small$-1$}\put(98,30){\small$-1$}\put(120,30){\small$0$}\put(128,30){\small$-1$}\put(143,30){\small$-1$}
\put(158,30){\small$-1$}\put(173,30){\small$-1$}\put(188,30){\small$-1$}\put(210,30){\small$0$}\put(225,30){\small$0$}

\put(15,40){\small$0$}\put(30,40){\small$0$}\put(38,40){\small$-1$}\put(53,40){\small$-1$}\put(68,40){\small$-1$}
\put(83,40){\small$-1$}\put(98,40){\small$-1$}\put(120,40){\small$0$}\put(128,40){\small$-1$}\put(143,40){\small$-1$}
\put(158,40){\small$-1$}\put(173,40){\small$-1$}\put(188,40){\small$-1$}\put(210,40){\small$0$}\put(225,40){\small$0$}

\put(15,50){\small$0$}\put(30,50){\small$0$}\put(38,50){\small$-1$}\put(53,50){\small$-1$}\put(68,50){\small$-1$}
\put(83,50){\small$-1$}\put(98,50){\small$-1$}\put(120,50){\small$0$}\put(128,50){\small$-1$}\put(143,50){\small$-1$}
\put(158,50){\small$-1$}\put(173,50){\small$-1$}\put(188,50){\small$-1$}\put(210,50){\small$0$}\put(225,50){\small$0$}

\put(15,60){\small$0$}\put(30,60){\small$0$}\put(38,60){\small$-1$}\put(53,60){\small$-1$}\put(68,60){\small$-1$}
\put(83,60){\small$-1$}\put(98,60){\small$-1$}\put(120,60){\small$0$}\put(128,60){\small$-1$}\put(143,60){\small$-1$}
\put(158,60){\small$-1$}\put(173,60){\small$-1$}\put(188,60){\small$-1$}\put(210,60){\small$0$}\put(225,60){\small$0$}

\put(15,70){\small$1$}\put(30,70){\small$1$}\put(45,70){\small$0$}\put(60,70){\small$0$}\put(75,70){\small$0$}
\put(90,70){\small$0$}\put(105,70){\small$0$}\put(120,70){\small$1$}\put(135,70){\small$0$}\put(150,70){\small$0$}
\put(165,70){\small$0$}\put(180,70){\small$0$}\put(195,70){\small$0$}\put(210,70){\small$1$}\put(225,70){\small$1$}
\put(122,73){\circle{12}}

\put(15,80){\small$0$}\put(30,80){\small$0$}\put(38,80){\small$-1$}\put(53,80){\small$-1$}\put(68,80){\small$-1$}
\put(83,80){\small$-1$}\put(98,80){\small$-1$}\put(120,80){\small$0$}\put(128,80){\small$-1$}\put(143,80){\small$-1$}
\put(158,80){\small$-1$}\put(173,80){\small$-1$}\put(188,80){\small$-1$}\put(210,80){\small$0$}\put(225,80){\small$0$}
\put(36,78){\line(1,0){76}}\put(112,78){\line(0,1){50}}\put(112,128){\line(-1,0){76}}\put(36,128){\line(0,-1){50}}
\put(126,78){\line(1,0){76}}\put(202,78){\line(0,1){50}}\put(202,128){\line(-1,0){76}}\put(126,128){\line(0,-1){50}}

\put(15,90){\small$0$}\put(30,90){\small$0$}\put(38,90){\small$-1$}\put(53,90){\small$-1$}\put(68,90){\small$-1$}
\put(83,90){\small$-1$}\put(98,90){\small$-1$}\put(120,90){\small$0$}\put(128,90){\small$-1$}\put(143,90){\small$-1$}
\put(158,90){\small$-1$}\put(173,90){\small$-1$}\put(188,90){\small$-1$}\put(210,90){\small$0$}\put(225,90){\small$0$}

\put(15,100){\small$0$}\put(30,100){\small$0$}\put(38,100){\small$-1$}\put(53,100){\small$-1$}\put(68,100){\small$-1$}
\put(83,100){\small$-1$}\put(98,100){\small$-1$}\put(120,100){\small$0$}\put(128,100){\small$-1$}\put(143,100){\small$-1$}
\put(158,100){\small$-1$}\put(173,100){\small$-1$}\put(188,100){\small$-1$}\put(210,100){\small$0$}\put(225,100){\small$0$}

\put(15,110){\small$0$}\put(30,110){\small$0$}\put(38,110){\small$-1$}\put(53,110){\small$-1$}\put(68,110){\small$-1$}
\put(83,110){\small$-1$}\put(98,110){\small$-1$}\put(120,110){\small$0$}\put(128,110){\small$-1$}\put(143,110){\small$-1$}
\put(158,110){\small$-1$}\put(173,110){\small$-1$}\put(188,110){\small$-1$}\put(210,110){\small$0$}\put(225,110){\small$0$}

\put(15,120){\small$0$}\put(30,120){\small$0$}\put(38,120){\small$-1$}\put(53,120){\small$-1$}\put(68,120){\small$-1$}
\put(83,120){\small$-1$}\put(98,120){\small$-1$}\put(120,120){\small$0$}\put(128,120){\small$-1$}\put(143,120){\small$-1$}
\put(158,120){\small$-1$}\put(173,120){\small$-1$}\put(188,120){\small$-1$}\put(210,120){\small$0$}\put(225,120){\small$0$}

\put(15,130){\small$1$}\put(30,130){\small$1$}\put(45,130){\small$0$}\put(60,130){\small$0$}\put(75,130){\small$0$}
\put(90,130){\small$0$}\put(105,130){\small$0$}\put(120,130){\small$1$}\put(135,130){\small$0$}\put(150,130){\small$0$}
\put(165,130){\small$0$}\put(180,130){\small$0$}\put(195,130){\small$0$}\put(210,130){\small$1$}\put(225,130){\small$1$}

\put(15,140){\small$1$}\put(30,140){\small$1$}\put(45,140){\small$0$}\put(60,140){\small$0$}\put(75,140){\small$0$}
\put(90,140){\small$0$}\put(105,140){\small$0$}\put(120,140){\small$1$}\put(135,140){\small$0$}\put(150,140){\small$0$}
\put(165,140){\small$0$}\put(180,140){\small$0$}\put(195,140){\small$0$}\put(210,140){\small$1$}\put(225,140){\small$0$}
\put(222,138){\line(1,0){11}}\put(233,138){\line(0,1){10}}\put(233,148){\line(-1,0){11}}\put(222,148){\line(0,-1){10}}
\end{picture}
\end{center}
Then, using \ref{HS}, one can read off the lattice cohomology from the picture: the big frames illustrate
the generators of $H^0(S_{-1},\Z)$, the small frames mark the generators of $H^0(S_{0},\Z)$
appeared at degree $0$ and the circle shows the generator of $H^1(S_{0},\Z)$. Hence,
$$\bH^0(M,k_{can})=\calt^+_{-2}\oplus \calt^3_{-2}(1)\oplus \calt^2_{0}(1) \ \ \ \mbox{and} \ \ \
\bH^1(M,k_{can})=\calt_0(1).$$

\end{document}